\documentclass[11pt,twoside,reqno]{amsart}
\usepackage[latin1]{inputenc}
\usepackage[OT1]{fontenc}
\usepackage{amsmath}
\usepackage{amsthm}
\usepackage{amssymb}
\usepackage[all]{xy}
\usepackage{bbm}
\usepackage{amscd}
\usepackage{a4wide}
\usepackage{enumitem}
\usepackage{stmaryrd}
\usepackage{mathtools}

\setenumerate[0]{label=(\roman*)}

\DeclareMathOperator{\Spec}{\mathsf{Spec}}

\DeclareMathOperator{\pr}{\mathsf{pr}}

\DeclareMathOperator{\id}{\mathsf{id}}
\DeclareMathOperator{\sgn}{\mathsf{sgn}}

\DeclareMathOperator{\sHom}{\mathcal{H}\textit{om}}
\DeclareMathOperator{\Ext}{\mathsf{Ext}}
\DeclareMathOperator{\End}{End}

\DeclareMathOperator{\Coh}{\mathsf{Coh}}

\DeclareMathOperator{\Ho}{\mathsf H}

\DeclareMathOperator{\Quot}{\mathsf{Quot}}
\DeclareMathOperator{\Flag}{\mathsf{Flag}}

\DeclareMathOperator{\rk}{\mathsf{rk}}

\DeclareMathOperator{\supp}{\mathsf{supp}}

\DeclareMathOperator{\triv}{\mathsf{triv}}

\newcommand{\llb}{\llbracket}
\newcommand{\rrb}{\rrbracket}

\newcommand{\llp}{\llparenthesis}
\newcommand{\rrp}{\rrparenthesis}

\newcommand{\wcQ}{\widetilde{ \mathcal Q}}
\newcommand{\wcK}{\widetilde{ \mathcal K}}

\newcommand{\D}{{\mathsf D}}

\DeclareMathOperator{\Aut}{\mathsf{Aut}}

\DeclareMathOperator{\Pic}{\mathsf{Pic}}

\newcommand{\cH}{{\mathcal H}}

\newcommand{\pt}{\mathsf{pt}}

\newcommand{\IC}{\mathbb{C}}

\newcommand{\IN}{\mathbb{N}}
\newcommand{\IP}{\mathbb{P}}

\newcommand{\IZ}{\mathbb{Z}}

\DeclareMathOperator{\VB}{\mathsf{VB}}
\DeclareMathOperator{\FM}{\mathsf{FM}}

\makeatletter
\newcommand{\leqnomode}{\tagsleft@true}
\newcommand{\reqnomode}{\tagsleft@false}
\makeatother

\let\ker\relax
\DeclareMathOperator{\ker}{\mathsf{ker}}

\let\dim\relax
\DeclareMathOperator{\dim}{\mathsf{dim}}

\let\min\relax
\DeclareMathOperator{\min}{\mathsf{min}}

\newcommand{\sym}{\mathfrak S}

\newcommand{\cF}{\mathcal F}
\newcommand{\cG}{\mathcal G}

\newcommand{\cE}{\mathcal E}

\newcommand{\cP}{\mathcal P}
\newcommand{\cQ}{\mathcal Q}

\newcommand{\cK}{\mathcal K}
\newcommand{\cS}{\mathcal S}

\newcommand{\reg}{\mathcal O}

\renewcommand{\theta}{\vartheta}
\renewcommand{\rho}{\varrho}
\renewcommand{\phi}{\varphi}
\renewcommand{\_}{\underline{\,\,\,\,}}

\usepackage[usenames,dvipsnames]{xcolor}
\usepackage{hyperref}
\hypersetup{
    colorlinks,
    linkcolor={red!50!black},
    citecolor={blue!50!black},
    urlcolor={blue!80!black}
}
\usepackage{aliascnt} 

\newtheorem{theorem}{Theorem}[section]

\newaliascnt{conjecture}{theorem}
\newtheorem{conjecture}[conjecture]{Conjecture}
\aliascntresetthe{conjecture}

  \newaliascnt{proposition}{theorem}
  \newtheorem{prop}[proposition]{Proposition}
  \aliascntresetthe{proposition}

  \newaliascnt{lemma}{theorem}
  \newtheorem{lemma}[lemma]{Lemma}
  \aliascntresetthe{lemma}

  \newaliascnt{corollary}{theorem}
  \newtheorem{cor}[corollary]{Corollary}
  \aliascntresetthe{corollary}

\theoremstyle{definition}

  \newaliascnt{definition}{theorem}
  
  \aliascntresetthe{definition}

  \newaliascnt{remark}{theorem}
  \newtheorem{remark}[remark]{Remark}
  \aliascntresetthe{remark}

  \newaliascnt{condition}{theorem}
  
  \aliascntresetthe{condition}

\newaliascnt{convention}{theorem}
 \newtheorem{convention}[convention]{Convention}
\aliascntresetthe{convention}

  \newaliascnt{question}{theorem}
  
  \aliascntresetthe{question}

  \newaliascnt{example}{theorem}
  
  \aliascntresetthe{example}

\begin{document}

\title[Tautological Bundles on Punctual Quot Schemes of Curves]{Extension Groups of Tautological Bundles on Punctual Quot Schemes of Curves}
\author[A. Krug]{Andreas Krug}
\address{
Institut f\"ur algebraische Geometrie,
Gottfried Wilhelm Leibniz Universit\"at Hannover,
Welfengarten 1,
30167 Hannover,
Germany
}
\email{krug@math.uni-hannover.de}

 \maketitle

\begin{abstract}
We prove formulas for the cohomology and the extension groups of tautological bundles on punctual Quot schemes over complex smooth projective curves. As a corollary, we show that the tautological bundle determines the isomorphism class of the original vector bundle on the curve. We also give a vanishing result for the push-forward along the Quot--Chow morphism of tensor and wedge products of duals of tautological bundles.   
\end{abstract}

\section{Introduction}

Tautological bundles on symmetric powers of curves have been intensively studied for more than half a century; see \cite{Schwarzenberger--bundlesplane, Schwarzenberger--secant, Mattuck--sym}. Given a smooth projective curve $C$ over the complex numbers $\IC$, a positive integer $d\in \IN$ and a vector bundle $F\in \VB(C)$, the associated tautological bundle $F^{[n]}$ on $C^{(d)}=C^d/\sym_d$ is the rank $d\cdot \rk F$ bundle with fibres 
\[
 F^{[d]}(x_1+\dots x_d)=\Ho^0(F_{\mid x_1+\dots +x_d})
\]
where $x_1+\dots+ x_d$ stands for both, a point in $C^{(d)}$ and the associated effective divisor on $C$. More formally, $F^{[d]}$ is constructed via the Fourier--Mukai transform along a universal family; see \autoref{subsect:SdC}. The reason for the interest in these bundles is at least twofold. On the one hand, they play an important role in the geometry of the symmetric product $C^{(d)}$. In particular, the cotangent bundle is given by $\Omega_{C^{(d)}}\cong (\Omega_C)^{[d]}$; see \cite[(9.1)]{Macdonald--sym}. On the other hand, a lot of information on the curve $C$ and vector bundles on it can be extracted from the tautological bundles; see, for example, the proof of the Gonality Conjecture \cite{Ein-Lazarsfeld--gonalityconj}. One result concerning tautological bundles that is probably well-known, and which we will use later in this introduction, is the following formula for their cohomology (see \autoref{prop:cohtautSdC}):
\begin{equation}\label{eq:tautSnCH}
 \Ho^*\bigl(C^{(d)}, F^{[d]}\bigr)\cong \Ho^*(F)\otimes S^{d-1}\Ho^*(\reg_C) \,.
\end{equation}
In recent years, there has been a lot of research concerning punctual Quot schemes; see e.g.\ \cite{MR4220750, MR4372634, Gango-Sebastian, MR4228267, Toda--Quotsod}, as well as the references that will play a role in the further course of this introduction. Given a fixed vector bundle $E\in \VB(C)$, the Quot scheme $\Quot_d(E)$ parametrises quotients $E\twoheadrightarrow Q$ where $Q$ is a zero-dimensional coherent sheaf of length $\ell(Q)=d$. There is the \emph{Quot--Chow morphism} sending a zero-dimensional quotient sheaf to its weighted support 
\[
 \mu\colon \Quot_d(E)\to C^{(d)}\quad,\quad \mu\left([E\twoheadrightarrow Q]\right)=\sum_{x\in \supp Q}\ell(Q_x)\,. 
\]
For $E=L$ a line bundle, $\mu$ is an isomorphism, so $\Quot_d(L)$ gives nothing new. Hence, from now on, we make the general assumption that
\begin{equation}\label{eq:rankassumption}
 \rk E\ge 2\,,
\end{equation}
which will also be necessary for most of our results. The Quot scheme $\Quot_d(E)$ is always smooth of dimension $d\cdot \rk(E)$; see \cite[Lem.\ 2.2]{BFP} or \cite{Monavari-Ricolfi--lisse}. In analogy to the symmetric products, every vector bundle $F\in \VB(C)$ induces a \emph{tautological} vector bundle $F^{\llb d\rrb}$ of rank $d\cdot \rk F$ on $\Quot_d(E)$ with fibres 
\[
 F^{\llb d\rrb}\bigl([E\twoheadrightarrow Q]\bigr)= \Ho^0(F_{\mid Q})\,. 
\]
Again, the formal definition uses the Fourier--Mukai transform along the universal family; see \autoref{subsect:Quottaut}. Oprea and Sinha \cite{Oprea-Sinha--Euler} computed formulas for Euler characteristics of many tensor and wedge products of tautological bundles, and their duals, associated to line bundles on $C$.
The author expects these formulas to lift to the individual cohomology groups as follows (see \cite[Quest.\ 20]{Oprea-Sinha--Euler} for a more modest formulation).

\begin{conjecture}\label{conj:main}
 Let $L, M_1,\dots, M_m\in\Pic C$ with $0\le m< \rk E$. Then, for $d\in \IN$, $0\le \ell\le d$, and $0\le k_i\le d$, we have
 \begin{equation}\label{eq:mainconj}
  \Ext^*\left(\bigotimes_{i=1}^m \wedge^{k_i} M_i^{\llb d \rrb}, \wedge^\ell L^{\llb d\rrb}\right)\cong \left(\bigotimes_{i=1}^m S^{k_i}\Ext^*(M_i, L)   \right)\otimes \left(\wedge^{\ell-k} \Ho^*(E\otimes L)\right)\otimes\left(S^{d-\ell}\Ho^*(\reg_C)\right)
 \end{equation}
where $k:=\sum_{i=1}^m k_i$.
\end{conjecture}
The symmetric and wedge powers in \eqref{eq:mainconj} are the graded ones, taking into account the odd and even parts of the cohomology groups; see e.g.\ \cite[Sect.\ 2.2]{Krug--curveExt} for details. Also, a wedge power with a negative exponent is interpreted as $0$. This means that, for $k=\sum_{i=1}^m k_i>\ell$, \autoref{conj:main} predicts the vanishing of
the whole $\Ext^*\bigl(\bigotimes_{i=1}^m \wedge^{k_i} M_i^{\llb d \rrb}, \wedge^\ell L^{\llb d\rrb}\bigr)$.

In the special case that $C\cong \IP^1$ is the projective line and $E\cong \reg_{\IP^1}^{\oplus \rk E}$, parts of the $\ell=0$ case and all of the $m=0$ case of \autoref{conj:main} were proven in \cite{MOS}. The case $m=0$ gives a formula for the cohomology for all wedge powers of $L^{\llb d\rrb}$.    

\subsection*{Summary of results}
In this paper, we work with an arbitrary smooth projective curve $C$ and an arbitrary $E\in \VB(C)$ of $\rk E\ge 2$. We prove the case $\ell=0$ (for arbitrary $m$, $M_i$ and $k_i$) of \autoref{conj:main}, where it predicts a vanishing of the extension groups; see \eqref{eq:cohdualformula}, \eqref{eq:Ldualvanish}. Furthermore, we prove the case $(m,\ell)=(0,1)$, giving the formula \eqref{eq:cohformula} for the cohomology of a tautological bundle, and the case $(m,\ell)=(1,1)$, giving the formula \eqref{eq:Extformula} for the extension groups between two tautological bundles. Our formulas hold more generally for tautological bundles  associated to vector bundles of arbitrary rank, not just line bundles. In the cases $(m,\ell)=(0,1)$ and $(m,\ell)=(1,1)$ this generalisation just means that we are allowed to replace the line bundles in \eqref{eq:mainconj} by arbitrary vector bundles. In the $\ell=0$ case, the generalisation to arbitrary rank is more subtle; see \autoref{thm:maindualvanish} and \autoref{rem:numcondition}.

Our first result describes the relationship of tautological bundles on the Quot scheme and tautological bundles on the symmetric product. It leads to the formula for the cohomology of the tautological bundles on $\Quot_d(E)$. 
\begin{theorem}[\autoref{thm:tautmupush} \& \autoref{cor:tautcoh}]\label{thm:maincoh}
For every $F\in \VB(C)$ and every $d\in \IN$,
\[
 R\mu_*F^{\llb d\rrb}\cong (E\otimes F)^{[d]}\,.
\]
This means that $\mu_*F^{\llb d\rrb}\cong (E\otimes F)^{[d]}$ and $R^i\mu_*F^{\llb d\rrb}\cong 0$ for $i>0$.
In particular, \eqref{eq:tautSnCH} gives 
\begin{equation}\label{eq:cohformula}
\Ho^*\bigl(\Quot_d(E), F^{\llb d\rrb}\bigr)\cong \Ho^*(E\otimes F)\otimes S^{d-1}\Ho^*(\reg_C) \,.
\end{equation}
\end{theorem}

Biswas, Gangopadhyay, and Sebastian \cite{BGS} computed the cohomology of the tangent bundle of the punctual Quot scheme. \autoref{thm:maincoh} is a rather easy corollary of one of their auxiliary results, describing the push-forward of the universal quotient sheaf along the Quot--Chow morphism.

Next, we give a quite general vanishing result for products of duals of tautological bundles.

\begin{theorem}[\autoref{thm:dualvanish}]\label{thm:maindualvanish}
 Let $F_1,\dots, F_m\in \VB(C)$, and let $1\le k_i\le d\cdot\rk F_i=\rk F_i^{\llb d \rrb}$ such that $\sum_{i=1}^m\min\{k_i, \rk F_i\}<\rk E$.
Then, for all $d\in \IN$,
\[
 R\mu_*\left(\bigotimes_{i=1}^m\wedge^{k_i} F_i^{\llb d\rrb\vee}\right)\cong 0\,.
\]
In particular,
\begin{equation}\label{eq:cohdualformula}
 \Ho^*\Bigl(\Quot_d(E),\bigotimes_{i=1}^m\wedge^{k_i} F_i^{\llb d\rrb\vee}\Bigr)\cong 0\,.
\end{equation}
\end{theorem}

Let us highlight two special cases, namely those where all $\min\{k_i, \rk F_i\}=1$, either since $k_i=1$ or since $\rk F_i=1$.

\begin{cor}\label{cor:dualvanish}
Let $1\le m< \rk E$.
\begin{enumerate}
 \item For any $F_1,\dots , F_m\in \VB(C)$, we have
 $R\mu_*\bigl(\bigotimes_{i=1}^m F_i^{\llb d\rrb\vee}\bigr)\cong 0$.
In particular, for a single vector bundle $F\in \VB(C)$, we have $R\mu_*(F^{\llb d\rrb \vee})\cong 0$.
\item For any $M_1,\dots , M_m\in \Pic(C)$, we have
$R\mu_*\bigl(\bigotimes_{i=1}^m\wedge^{k_i} M_i^{\llb d\rrb\vee}\bigr)\cong 0$
for all $1\le k_i\le d$. In particular, we get the $\ell=0$ case of \autoref{conj:main}, namely
\begin{equation}\label{eq:Ldualvanish}
 \Ho^*\left(\Quot_d(E),\bigotimes_{i=1}^m\wedge^{k_i} M_i^{\llb d\rrb\vee}\right)\cong 0\,.
\end{equation}
 \end{enumerate}
\end{cor}

Probably our main result is a formula for the push-forwards of the Hom-bundles between two tautological bundles, and a resulting formula for the extension groups.

\begin{theorem}[\autoref{thm:muHom}]\label{thm:mainExt}
For every $d\in \IN$ and all $F,G\in \VB(C)$, we have
\[
 R\mu_*\sHom(F^{\llb d\rrb}, G^{\llb d\rrb})\cong \sHom(F,G)^{[d]}
\]
In particular, due to \eqref{eq:tautSnCH}, we have
\begin{equation}\label{eq:Extformula}
\Ext^*(F^{\llb d\rrb}, G^{\llb d\rrb})\cong \Ext^*(F,G)\otimes S^{d-1}\Ho^*(\reg_C) \,.
\end{equation}
\end{theorem}

Without assumption \eqref{eq:rankassumption} that $\rk E\ge 2$, \autoref{thm:mainExt} does not hold. Indeed, if $E=L$ is a line bundle, then $\mu$ is an isomorphism. But if, for example, $\rk F=\rk G=2$, then  
$\sHom(F^{\llb 2\rrb}, G^{\llb 2\rrb})$ has rank $16$, while $\sHom(F,G)^{[2]}$ only has rank $8$.

For tautological bundles on symmetric powers of curves (or, equivalently, on Quot schemes of quotients of line bundles), the extension groups between tautological bundles were described in \cite{Krug--curveExt} by means of a spectral sequence. However, there is not a simple closed formula relating the extension groups on $C^{(d)}$ to those of $C$. Instead, the extension groups on $C^{(d)}$ depend on more input data, namely certain cup and Yoneda products on $C$. 

Hence, maybe surprisingly, \autoref{thm:mainExt} shows that the situation becomes a lot easier under the assumption $\rk E\ge 2$. In fact, \eqref{eq:Extformula} looks closer to the result for Hilbert schemes of points on \emph{surfaces}, instead of curves; see \cite{KruExt} (or \cite{Krug--remarksMcKay} for a simplified proof of the result).
See also \cite[Sect.\ 1.2]{Oprea-Sinha--Euler} for some further parallels between punctual Quot schemes on curves and Hilbert schemes of points on surfaces.

We also get a result closely related to \autoref{thm:mainExt} in terms of the \emph{tautological functor}. This is the Fourier--Mukai transform $T_d=\FM_{\cQ_d}\colon \D(C)\to \D(\Quot_d)$ which on vector bundles $F\in \VB(C)$ is given by $T_d(F)\cong F^{\llb d \rrb}$.

\begin{theorem}[\autoref{thm:convolution}]\label{thm:mainfunctor}
Let $R_d\colon \D(\Quot_d)\to \D(C)$ be the right-adjoint of $T_d$. Then
\[
R_d\circ T_d \cong (\_)\otimes_{\IC} S^{d-1}\Ho^*(\reg_C)\,.
\]
\end{theorem}

As a simple corollary, we get the following reconstruction result.

\begin{cor}
For any two vector bundles $F,G\in \VB(C)$ and any $d\in \IN$, we have 
\[
F^{\llb d\rrb}\cong G^{\llb d\rrb} \quad\Longrightarrow \quad F\cong G\,. 
\]
\end{cor}

\begin{proof}
To the given isomorphism $T_d(F)\cong F^{\llb d\rrb}\cong G^{\llb d\rrb} \cong T_d(G)$, we apply $R_d$. By \autoref{thm:mainfunctor}, this gives 
\[
F\otimes_{\IC} S^{d-1}\Ho^*(\reg_C)\cong G\otimes_{\IC} S^{d-1}\Ho^*(\reg_C) 
\]
As $\Ho^0(\reg_C)\cong \IC$, looking at the degree zero part (i.e.\ applying $\cH^0(\_)$) gives $F\cong G$.
\end{proof}

On $C^{(d)}$, it depends on the curve $C$ whether or not the isomorphism class of $F$ is determined by $F^{[n]}$, with a full answer not known yet for $g(C)\ge 2$; see \cite[Discussion below Thm.\ 1.3]{Krug--reconstruction}, \cite[Prop.\ 2.1]{Biswas-Nagaraj--reconstructionsurfaces}. So we observe that, again, the situation becomes easier for $\rk E\ge 2$.

\subsection*{Strategy of the Proofs}

Our approach to proving \autoref{thm:maindualvanish}, \autoref{thm:mainExt}, and \autoref{thm:mainfunctor} is inspired by an approach that was very fruitful in the study of tautological bundles on symmetric products of curves \cite{Mistretta--stab, Krug--stab, Krug--curveExt}, and also on Hilbert schemes of points on surfaces \cite{Sca1, Sca2, MR4074590, Stapletontaut, KruExt, KrugTensortaut, Krug--remarksMcKay}: Translate everything to (equivariant) sheaves on the cartesian product, where the computations often become easier. In the curve case, this is simply done by the pull-back along the quotient morphism $\pi\colon C^{(d)}\to C^{[d]}$. In the surface case, one uses the derived McKay correspondence $\D(X^{[d]})\cong \D_{\sym_d}(X^d)$ of Bridgeland--King--Reid \cite{BKR} and Haiman \cite{Hai}.

A first idea to handle the tautological bundles on $\Quot_d(E)$ in a similar way would be to use the fibre product $Y_d:=\Quot_d(E)\times_{C^{(d)}} C^d$. However, we do not have a good description of $Y_d$, neither of the pull-backs of tautological sheaves to it. The successful approach is to use $\Flag_d(E)$, the variety parametrising full flags of zero-dimensional quotients
\begin{equation}\label{eq:fullflagintro}
 E\twoheadrightarrow Q_d \twoheadrightarrow Q_{d-1}\twoheadrightarrow \dots \twoheadrightarrow Q_1 \twoheadrightarrow Q_0=0
\end{equation}
with $\ell(Q_i)=i$. As explained at the beginning of \autoref{sect:McKay}, our use of $\Flag_d(E)$ can de regarded as a global version of the approach used in \cite{BGS} to study the fibres of the Quot--Chow morphism $\mu\colon \Quot_d(E)\to C^{(d)}$. There is a \emph{Flag--Quot} morphism $\Flag_d(E)\to \Quot_d(E)$, sending a flag \eqref{eq:fullflagintro} to its top quotient $Q_d$, and also a morphism $\Flag_d(E)\to C^d$ sending \eqref{eq:fullflagintro} to the ordered tuple $\bigr(\supp(\ker(Q_i\twoheadrightarrow Q_{i-1}) \bigl)_{i=1}^d$ of the supports of its factors. The reason that the approach using the Flag variety works well is:
\begin{enumerate}
 \item There is a simple construction of $\Flag_d(E)$ as an iterated projective bundle; see \autoref{sect:McKay}.
 \item There is a simple description, perfectly suited for induction arguments, of the pull-backs of tautological sheaves to $\Flag_d(E)$; see \autoref{lem:tautses}.
 \item We do not lose much information if we pull back sheaves along the Flag--Chow morphism to $\Flag_d(E)$; see \autoref{cor:gfaith} and \autoref{prop:McKay}.
\end{enumerate}
The author is quite optimistic that the approach to pull back everything to $\Flag_d(E)$ should eventually lead to a proof of all of \autoref{conj:main}. However, the cases with $\ell\ge 2$ seem to require more complicated computations.  

\subsection*{Structure of the Paper}  

In \autoref{sect:basics}, we recall some basic definitions and facts on symmetric powers and punctual Quot schemes 
of curves, and on tautological bundles on these varieties. In the short \autoref{sect:coh}, we deduce \autoref{thm:maincoh} from a result of \cite{BGS}. 

In \autoref{sect:McKay}, we explain the generalities of our approach using the Flag variety as described above. This is then used in \autoref{sect:dualvanish} to prove \autoref{thm:maindualvanish}. 

For our proof of \autoref{thm:mainExt} and \autoref{thm:mainfunctor}, we need some techniques concerning Fourier--Mukai functors and equivariant sheaves. These are recalled in \autoref{subsect:FM} and \autoref{subsect:equi}, respectively. After introducing some incidence schemes in \autoref{subsect:Xi} that will be needed, we do the main parts of the proofs of \autoref{thm:mainExt} and \autoref{thm:mainfunctor} in \autoref{subsect:Ext} using Fourier--Mukai computations.

In \autoref{sect:furtherrem}, we give two further observations. The first is that we can generalise the formulas \eqref{eq:cohformula}, \eqref{eq:cohdualformula}, and \eqref{eq:Extformula} a bit further by adding natural line bundles on the Quot scheme. The second is that \autoref{thm:mainfunctor} gives a part of a semi-orthogonal decomposition in the case that $C=\IP^1$ is the projective line.

\subsection*{Notation and conventions}

We will work over the complex numbers $\IC$ as our ground field.

Given a smooth projective variety $X$, its bounded derived category of coherent sheaves is denoted by $\D(X):=\D^b(\Coh(X))$. In \autoref{sect:basics}, \ref{sect:coh}, \ref{sect:McKay}, \ref{sect:dualvanish} we will make it explicit in our notation whenever we are using derived functors, writing, for example, the derived push-forward along the Quot--Chow morphism by $R\mu_{d*}$, as we did in the introduction. However, once we start doing heavy Fourier--Mukai calculations in \autoref{sect:FM}, in order to shorten the formulas a bit, we will tacitly assume that every functor is derived. For example, the derived push-forward will then be denoted by $\mu_{d*}\colon \D(\Quot_d(E))\to \D(C^{(d)})$; see \autoref{conv:derived}.    

We will often consider projection morphisms from products of varieties to some of their factors. These will always be denoted by $\pr$ with the numbering of the factors to which we project in the lower index. For example, we will always write
\[
 \pr_{12}\colon X\times Y\times Z\to X\times Y\quad,\quad (x,y,z)\mapsto (x,y)\,,
\]
independently of which factors $X$, $Y$, $Z$ we actually have. The advantage is that we do not have to introduce a new letter whenever yet another projection morphism shows up. The disadvantage is that, for example, $\pr_{2}$ will stand for several different morphisms; see e.g.\ diagram \eqref{eq:simplediag} below. We found it unlikely that this ambiguity will cause serious problems for the reader. 

If the numbers in the lower index of $\pr$ are not in the standard order, this denotes the projection to the factors, followed by the morphism interchanging the order of the factors. For example, $\pr_{32}\colon X\times Y\times Z\to Z\times Y$ denotes the morphism given by $\pr_{32}(x,y,z)=(z,y)$.

\section{Summary of Basics on Quot Schemes and Tautological Bundles}\label{sect:basics}

\subsection{Symmetric powers of curves and their tautological bundles}\label{subsect:SdC}

Throughout the text, $C$ will denote a smooth projective curve. For any $d\in \IN$, the \emph{$d$-th symmetric power of $C$} is the quotient $C^{(d)}:=C^d/\sym_d$ of the cartesian power $C^d$ by the symmetric group $\sym_d$.
We denote the quotient morphism by $\pi=\pi_d\colon C^d\to C^{(d)}$, and record the following well-known fact.
\begin{lemma}\label{lem:piflat}
 For every $d\in \IN$, the quotient morphism $\pi_d\colon C^d\to C^{(d)}$ is flat and $C^{(d)}$ is smooth. 
\end{lemma}
\begin{proof}
 This can be deduced from the Fundamental Theorem on Symmetric Functions, or from the Chevalley--Shephard--Todd Theorem.
\end{proof}

We write the closed points of $C^{(d)}$ as formal sums $x_1+\dots+x_d:=\pi(x_1,\dots,x_d)$.
Indeed, the symmetric power can be regarded as the moduli space of effective divisors of degree $d$ (or as the Hilbert scheme of $d$ points) with the universal family $\Xi=\Xi_d\subset C\times C^{(n)}$ being the image of the closed embedding
\begin{equation}\label{eq:Xiembedding}
C\times C^{(d-1)}\hookrightarrow C\times C^{(d)}\quad,\quad (x,x_1+\dots+x_{d-1})\mapsto (x, x_1+\dots+x_{d-1}+x)\,.  
\end{equation}
Let $a\colon \Xi\to C$ and $b\colon \Xi\to C^{(d)}$ be the restriction of the projections $\pr_1\colon C\times C^{(d)}\to C$ and $\pr_2\colon C\times C^{(d)}\to C^{(d)}$, respectively. Since $b$ is flat and finite of degree $d$, for any vector bundle $F\in \VB(C)$, we get an induced \emph{tautological bundle} 
\[
 F^{[d]}:=b_*a^*F\in \VB(C^{(d)})
\]
with $\rk F^{[d]}=d\cdot\rk F$.

\begin{prop}\label{prop:cohtautSdC}
For every $d\in \IN$, we have 
\[
\Ho^*\bigl(C^{(d)}, F^{[d]}\bigr)\cong \Ho^*(F)\otimes S^{d-1}\Ho^*(\reg_X)\,.
\]
\end{prop}

\begin{proof}
This is easy and most likely well-known. For the global sections in the case that $F$ is a line bundle, it was already stated in \cite[Cor.\ to Prop.\ 1]{Mattuck--sym}. However, we did not find a reference for the general statement, so here is the proof: As $b$ is finite, we have
\[
 \Ho^*(C^{(d)}, F^{[d]})\cong \Ho^*(C^{(d)}, b_*a^*F)\cong \Ho^*(\Xi_d, a^*F)\,. 
\]
Under the isomorphism $C\times C^{(d-1)}\cong \Xi_d$ described in \eqref{eq:Xiembedding}, $a^*F$ corresponds to the box product $F\boxtimes \reg_{C^{(d-1)}}$. Hence, by the K\"unneth isomorphism
\begin{align*}
\Ho^*(\Xi_d, a^*F)\cong \Ho^*(F)\otimes \Ho^*(\reg_{C^{(d-1)}})&\cong \Ho^*(F)\otimes \Bigl(\Ho^*(\reg_C)^{\otimes d-1}\Bigr)^{\sym_{d-1}}\\&\cong \Ho^*(F)\otimes S^{d-1}\Ho^*(\reg_X)\,. \qedhere
\end{align*}
\end{proof}

\subsection{Punctual Quot schemes and their tautological bundles}\label{subsect:Quottaut}


Given some vector bundle $E\in \VB(C)$, a natural generalisation of the symmetric power of the curve is given by the \emph{punctual Quot scheme} $\Quot_d(E)$, which is the moduli space of zero-dimensional length $d$ quotients $E\twoheadrightarrow Q$.
There is the \emph{Quot--Chow morphism} which sends a length $d$ quotient to its weighted support
\[
 \mu=\mu_d\colon \Quot_d(E)\to C^{(d)}\quad,\quad Q\mapsto \sum_{x\in \supp Q} \ell(Q_x)\,;
\]
see \cite[Sect.\ 2.1 \& 2.2]{BFP} for details on the construction of this morphism.

Let us collect some important basic facts on the geometry of the punctual Quot scheme:

\begin{prop}\label{prop:Quotbasicinfo}
For every smooth projective curve $C$, every $E\in \VB(C)$, and every $d\in \IN$,
\begin{enumerate}
 \item $\Quot_d(E)$ is smooth of dimension $\dim \Quot_d(E)=d\cdot\rk E$,
 \item $\mu\colon \Quot_d(E)\to C^{(d)}$ is flat.
\end{enumerate}
\end{prop}

\begin{proof}
For (i), see \cite[Lem.\ 2.2]{BFP} or \cite{Monavari-Ricolfi--lisse}. Part (ii) is  \cite[Cor.\ 6.3]{Gango-Sebastian}. 
\end{proof}

From now on, the curve $C$ and the vector bundle $E\in \VB(C)$ with $\rk E\ge 2$ will be fixed, and we abbreviate
\[
 \Quot_d:=\Quot_d(E)\,.
\]
We denote the universal quotient sequence on $C\times \Quot_d$ by 
\[
 0\to \cK_d\to E\boxtimes \reg_{\Quot_d}\to \cQ_d\to 0\,.
\]
As $\cQ$ is flat and finite of degree $d$ over $\Quot_d$, to any vector bundle $F\in \VB(C)$, we can again associated a \emph{tautological bundle} with $\rk F^{\llb d\rrb}=d\cdot \rk F$ by 
\[
 F^{\llb d \rrb}:=\pr_{2*}(\pr_1^*F\otimes \cQ_d)\in \VB(\Quot_d)\,.
\]

\section{Cohomology of tautological bundles}\label{sect:coh}  

In \cite{BGS}, the cohomology of the tangent bundle of $\Quot_d$ was computed. The tangent bundle is not a tautological bundle. Still, one of the auxiliary results in \emph{loc.\ cit}.\ leads to a formula or the cohomology of tautological bundles, as we explain in this short section.

\begin{prop}[{\cite[Cor.\ 9.3]{BGS}}]\label{prop:BGS}
We have $R(\id_C\times \mu)_*\cQ_d\cong \bigl(E\boxtimes\reg_{C^{(d)}}\bigr)_{\mid \Xi_d}$.
\end{prop}

\autoref{thm:maincoh} follows easily from this.

\begin{theorem}\label{thm:tautmupush}
We have $R\mu_*F^{\llb d\rrb}\cong (E\otimes F)^{[d]}$. 
\end{theorem}

\begin{proof}
Considering the commutative diagram
\begin{equation}\label{eq:simplediag}
 \xymatrix{
&\ar_{\pr_1}[dl] C\times \Quot_d \ar^{\pr_2}[r] \ar^{\id\times \mu}[d] &  \Quot_d \ar^{\mu}[d] \\
C  & \ar^{\pr_1}[l]  C\times C^{(d)} \ar^{\pr_2}[r] & C^{(d)}\,,
 }
\end{equation}
we get
\begin{align*}
R\mu_*F^{\llb d\rrb}
\cong &R\mu_*\pr_{2*}(\pr_1^* F\otimes \cQ)\\
\cong & \pr_{2*}R(\id\times \mu)_*(\pr_1^* F\otimes \cQ)\\
\cong & \pr_{2*}( \pr_1^* F\otimes R(\id\times \mu)_*\cQ) \tag{projection formula along $\id\times \mu$}\\
\cong & \pr_{2*}( \pr_1^* (E\otimes F)\otimes \reg_{\Xi_d}) \tag{\autoref{prop:BGS}}\\
\cong &  (E\otimes F)^{[d]}\,.\qedhere
\end{align*}
\end{proof}

\begin{cor}\label{cor:tautcoh}
 $\Ho^*(\Quot_d, F^{\llb d\rrb})\cong \Ho^*(E\otimes F)\otimes S^{d-1}\Ho^*(\reg_C)$.
\end{cor}

\begin{proof}
This is \autoref{thm:tautmupush} together with \autoref{prop:cohtautSdC}.
\end{proof}

\section{Kind of a McKay correspondence}\label{sect:McKay}

For $D=x_1+\dots+x_d\in C^{(d)}$, we write $\Quot_D$ for the fibre of $\mu\colon \Quot_d\to C^{(d)}$ over the point $D$. 
One key to obtain \autoref{prop:BGS} and related results in \cite{BGS} is to study the restriction of the universal sheaves on $C\times \Quot_d$ to the fibres of the Quot--Chow morphism by means of a certain birational morphism $g_{\vec x}\colon S_{\vec x}\to \Quot_D$ (note that in \cite{BGS}, and also in \cite{Gango-Sebastian} where it was studied first, this morphism is denoted by $g_d\colon S_d\to \cQ_D$ instead).
Concretely, given a point $D=x_1+\dots+x_d\in C^{(d)}$, an order $\vec x=(x_1,\dots,x_d)$ of the points of $D$ is chosen. Then $S_{\vec x}$ is constructed as the moduli space of full flags of quotients
\begin{equation}\label{eq:fullflag}
 E\twoheadrightarrow Q_d \twoheadrightarrow Q_{d-1}\twoheadrightarrow \dots \twoheadrightarrow Q_1 \twoheadrightarrow Q_0=0
\end{equation}
where $Q_i$ has length $\ell(Q_i)=i$ and weighted support $\mu_i(Q_i)=x_1+\dots+x_i$. This $S_{\vec x}$ is then an iterated projective bundle, which makes it quite easy to do concrete computations involving vector bundles on it.

In order to understand also the push-forwards of Hom bundles and products of tautological bundles, we globalise this approach. For this, we consider the moduli space $\Flag_d:=\Flag_d(E)$ of \textit{all} full flags of zero-dimensional quotients of $E$, i.e.\ of iterated quotients as in \eqref{eq:fullflag} with $\ell(Q_i)=i$ but \emph{without} any condition on the support of the $Q_i$.

Let us construct $\Flag_d$ inductively, along with a morphism $\nu_d\colon \Flag_d\to C^d$, which will be the morphism that sends a flag \eqref{eq:fullflag} to 
\[(x_1,x_2,\dots,x_d) \in C^d\quad,\quad x_i=\supp\bigl(\ker(Q_i\to Q_{i-1})\bigr)\,.\]
 For $d=0$, we have that $\Flag_0=\Quot_0=C^0$ is a point. For $d=1$, we have the $\IP$-bundle $\Flag_1=\Quot_1=\IP(E)\xrightarrow{\nu_1} C$. The universal length 1 quotient on $C\times \IP(E)$ is given by $E\boxtimes\reg_{\IP(E)}\twoheadrightarrow  i_{1*} \nu_1^*E\twoheadrightarrow i_{1_*}\reg_{\IP(E)}(1)$ where $i_1=(\nu_1,\id)$, the first surjection is the restriction map, and the second morphism is the push-forward of the tautological surjection $\nu_1^*E\twoheadrightarrow \reg_{\IP(E)}(1)$; see \cite[Prop.\ 2.6]{BFP} for details.

Now, assume that we have already constructed $\nu_{d-1}\colon \Flag_{d-1}\to C^{d-1}$ and on $C\times \Flag_{d-1}$ a universal length $d-1$ flag
\begin{equation}\label{eq:Flagd-1}
 E\boxtimes \reg_{\Flag_{d-1}}\twoheadrightarrow \cQ^{d-1}_{d-1}\twoheadrightarrow \dots \twoheadrightarrow \cQ^{d-1}_1 \twoheadrightarrow \cQ^{d-1}_0=0\,.
\end{equation}
In particular, denoting the flat family of length $d-1$ quotients by $\wcQ_{d-1}:=\cQ^{d-1}_{d-1}$, we have a short exact sequence sequence
\[
0\to \wcK_{d-1}\to E\boxtimes \reg_{\Flag_{d-1}}\to \wcQ_{d-1}\to 0\,. 
\]
The kernel $\wcK_{d-1}$ is again a vector bundle with $\rk\wcK_{d-1}=\rk E$ on $C\times \Flag_{d-1}$; see e.g.\ the argument for local freeness of $\cK$ in \cite[Sect.\ 2.2]{BFP}. 
We define $\Flag_d$ as the associated $\IP$-bundle
\begin{equation}\label{eq:p}
p=(p_1,p_2)\colon \Flag_d:=\IP(\wcK_{d-1})\to C\times \Flag_{d-1}\,.
\end{equation}
together with the morphism
\begin{equation}\label{eq:nufactor}
\nu_d=(\id_C\times \nu_{d-1})\circ p\colon \Flag_d\to C\times C^{d-1}=C^d\,.
\end{equation}     
To construct the universal flag, we write 
\[t=(\id,p_2)\colon C\times \Flag_d\to C\times \Flag_{d-1}\quad,\quad i=(p_1,\id)\colon \Flag_d\to C\times \Flag_d\,.\] 
By the flatness of $t$, we have the short exact sequence 
\begin{align*}
0\to t^*\wcK_{d-1}\to E\boxtimes \reg_{\Flag_{d}}\to t^*\wcQ_{d-1}\to 0\,. 
\end{align*}
Now, noting that $t\circ i=p$, we have a restriction map $t^*\wcK_{d-1}\twoheadrightarrow i_*i^*t^*\wcK_{d-1}\cong i_*p^*\wcK_{d-1}$. Composing this with the push-forward of the tautological surjection $p^*\wcK_{d-1}\twoheadrightarrow\reg_p(1)$ gives a surjection $t^*\wcK_{d-1}\twoheadrightarrow i_*\reg_p(1)$. We denote its kernel by $\wcK_d$ which means that we have a short exact sequence 
\[
 0\to \wcK_d\to t^*\wcK_{d-1}\to i_*\reg_p(1)\to 0\,.
\]
We denote the cokernel of the inclusion $\wcK_d\subset E\boxtimes \reg_{\Flag_d}$ by $\wcQ_d=\cQ^d_d$ which means
\[
0\to \wcK_{d}\to E\boxtimes \reg_{\Flag_{d}}\to \wcQ_{d}\to 0\,. 
\]
Comparing the above short exact sequences gives, via the snake lemma, another short exact sequence
\begin{equation}\label{eq:sesinduction}
0\to i_*\reg_p(1)\to \wcQ_d\to t^*\wcQ_{d-1}\to 0\,,
\end{equation}
which we will use heavily later for our induction arguments.
This $\wcQ_d=\cQ_d^d$ completes the pull-back of the universal full length $d-1$ flag \eqref{eq:Flagd-1} along $t$ to the universal full length $d$ flag which, setting $\cQ^d_i:=t^*\cQ^{d-1}_i$ for $i=0,\dots,d-1$, reads 
\begin{equation*}
 E\boxtimes \reg_{\Flag_{d}}\twoheadrightarrow \cQ_d^d \twoheadrightarrow \cQ^{d-1}_{d-1}\twoheadrightarrow \dots \twoheadrightarrow \cQ^{d-1}_1 \twoheadrightarrow \cQ^{d-1}_0=0\,.
\end{equation*}
Let $f_d\colon \Flag_d\to \Quot_d$ be the classifying morphism for $E\boxtimes \reg_{\Flag_{d}}\twoheadrightarrow \cQ_d^d=\wcQ_d$, which means $(\id_C\times f)^*\cQ_d\cong \wcQ_d$. We call $f_d$ the \emph{Flag--Quot morphism}.

Since $\mu_d\circ f_d=\pi_d\circ \nu_d$, we get an induced morphism $g_d\colon \Flag_d\to Y_d:= C^d\times_{C^{(d)}} \Quot_d$
to the fibre product making the following diagram commute
\begin{equation}\label{eq:Flagdiag}
 \xymatrix{
\Flag_d \ar@/^5mm/^{f_d}[rrd]  \ar@/_5mm/^{\nu_d}[ddr] \ar^{g_d}[dr] &  &    \\
 & Y_d \ar^{\mu_d'}[d] \ar^{\pi_d'}[r] &  \Quot_d \ar^{\mu_d}[d]  \\
 & C^d \ar^{\pi_d}[r] &  C^{(d)}\,.
 }
\end{equation}

In the following, in order to lighten the notation, we will often drop the indices $(\_)_d$ from the morphisms when they should be clear from the context.

\begin{lemma}\label{lem:nupush}
The variety $\Flag_d$ is smooth of dimension $d\cdot \rk E$. Furthermore, the morphisms $p\colon \Flag_d=\IP(\wcK_{d-1})\to C\times \Flag_{d-1}$ and $\nu\colon \Flag_d\to C^d$ have the following properties: 
\begin{align*}
Rp_*\reg_p(\ell)&\cong 0 \quad \text{for } -\rk E<\ell<0\,, \tag{i}\\
Rp_*\reg_p(1)&\cong \wcK_{d-1}\,, \tag{ii}\\ 
Rp_*\reg_{\Flag_d}&\cong \reg_{C\times \Flag_{d-1}}\,, \tag{iii}\\
Rp_*\circ p^*&\cong \id \colon \D(C\times\Flag_{d-1})\to \D(C\times\Flag_{d-1})\,, \tag{iv}\\
R\nu_*\reg_{\Flag_d}&\cong \reg_{C^d}\,.\tag{v}
\end{align*}
\end{lemma}

\begin{proof}
Everything follows from the iterated $\IP$-bundle construction of $\Flag_d$. Indeed, every time we move from $\Flag_{i-1}$ to $\Flag_i$, the dimension increases by $\rk E$. Namely, we get one extra dimension by adding the factor $C$, and $\rk E-1$ extra dimensions by forming the $\IP^{\rk E-1}$-bundle $p\colon \Flag_d=\IP(\wcK_{d-1})\to C\times \Flag_{d-1}$. 

Formulas (i), (ii), and (iii) hold for every projective bundle; see e.g.\ \cite[App.\ A]{Lazarsfeld--posbookI}. Formula (iv) follows from (iii) and the projection formula. Formula (v) follows from an iterative use of (iii). 
\end{proof}

\begin{lemma}\label{lem:gpush}
For every $d\in \IN$, we have $Rg_*\reg_{\Flag_d}\cong \reg_{Y_d}$.
\end{lemma}

\begin{proof}
Given $\vec x=(x_1,\dots,x_d)\in C^d$, the map $g_{\vec x}\colon S_{\vec x}\to \Quot_D$ is the fibre of $g$ over $\vec x \in C^d$, which means that we have a cartesian diagram
\[
\xymatrix{
S_{\vec x}\ar^{j'_{\vec x}}[d]  \ar^{g_{\vec x}}[r] & \Quot_D \ar^{j_{\vec x}}[d] \ar[r] & \{\vec x\} \ar[d]\\
\Flag_d \ar^g[r]   & Y_d \ar^{\mu'}[r] & C^d
}
\]
The map $g_{\vec x}$ is birational by \cite[Sect.\ 5]{Gango-Sebastian} and satisfies
\[Rg_{\vec x*}\reg_{S_{\vec x}}\cong \reg_{\Quot_D}\]
by \cite[Proof of Cor.\ 6.2]{BGS}.
In particular, $S_{\vec x}$ and $\Quot_D$ have the same dimension. Also $\Flag_d$ and $Y_d$ have the same dimension, namely $d\cdot \rk E$; see \autoref{lem:nupush} and \autoref{prop:Quotbasicinfo}(i). Hence, the closed embeddings $j_{\vec x}$ and $j'_{\vec x}$ have the same codimension.

As the morphism $\mu\colon \Quot_d\to C^{(d)}$ is flat by \cite[Cor.\ 6.3]{Gango-Sebastian} so is $\mu'$. It follows that $j_{\vec x}$ is a regular closed embedding (regular closed embeddings are stable under flat base change; see \cite[Prop.\ VIII 1.6]{SGA6}).
Hence, the assumptions of the base change theorem \cite[Cor.\ 2.27]{Kuz} are satisfied, which gives
\begin{equation}\label{eq:degzeroconc}
 Lj_{\vec x}^*Rg_*\reg_{\Flag_d}\cong Rg_{\vec x*} Lj_{\vec x}'^* \reg_{\Flag_d}\cong Rg_{\vec x*}\reg_{S_{\vec x}}\cong \reg_{\Quot_D}\,.
\end{equation}
In particular, $Lj_{\vec x}^*Rg_*\reg_{\Flag_d}$ is concentrated in degree 0 for all $\vec x\in C^d$. As the $j_{\vec x}(\Quot_D)$ for varying $\vec x\in C^n$ cover all of $Y_d$, it follows that also $Rg_*\reg_{\Flag_d}$ is concentrated in degree 0. Indeed, set $m:=\max\{i\mid R^ig_*\reg_{\Flag_d}\neq 0\}$. Then, we have some $\vec x\in C^d$ such that
\[
\supp\bigl( R^mg_*\reg_{\Flag_d}\bigr)\cap j_{\vec x}(\Quot_D)\neq \emptyset\,.
\]
For this $\vec x$, we have $L^0g_{\vec x}^*R^mg_*\reg_{\Flag_d}\neq 0$. The spectral sequence (see e.g.\ \cite[(3.10)]{Huy})
\[
 E_2^{p,q}=L^pj_{\vec x}^*R^pg_*\reg_{\Flag_d}\quad\Longrightarrow E^{n}= \cH^n\bigl(Lj_{\vec x}^*Rg_*\reg_{\Flag_d}\bigr)
\]
then shows that $m=0$, as $E^n\cong 0$ for $n\neq 0$ by \eqref{eq:degzeroconc}. Further studying the same spectral sequence, we see that, for all $\vec x\in C$,
\[
L^p j_{\vec x}^*g_*\reg_{\Flag_d}\cong \begin{cases}
             0\quad&\text{for $p<0$,}\\
             \reg_{\Quot_D}\quad&\text{for $p=0$.}
             \end{cases}
 \]
It follows that $g_*\reg_{\Flag_d}$ is a line bundle. But, if the push-forward of a structure sheaf is a line bundle, it is already the trivial line bundle (if $g_*\reg_{\Flag_d}$ is a line bundle, the morphism $\Spec_{Y_d}(g_*\reg_{\Flag_d})\to Y_d$ is an isomorphism).
\end{proof}

\begin{cor}\label{cor:gfaith}
For every $d$, we have an isomorphism of functors
\[
Rg_*\circ Lg^*\cong \id\colon \D(Y_d)\to \D(Y_d)\,.
\]
\end{cor}

\begin{proof}
 This follows from \autoref{lem:gpush} using the projection formula.
\end{proof}

\begin{cor}\label{prop:McKay}
For every $d$, we have an isomorphism of functors
\[
\pi^*\circ R\mu_{*}\cong R\nu_{*}\circ Lf^*\colon \D(\Quot_d)\to \D(C^d)\,.
\]
\end{cor}

\begin{proof}
Using flat base change along \eqref{eq:Flagdiag} and \autoref{cor:gfaith}, we get
\[
\pi^*\circ R\mu_*\cong R\mu'_*\circ \pi'^*\cong R\mu'_*\circ Rg_*\circ Lg^*\circ \pi'^*\cong R\nu_*\circ Lf^*\,.    \qedhere
\]
\end{proof}

\section{Vanishing of products of duals}\label{sect:dualvanish}

In this section, we prove \autoref{thm:maindualvanish}. Using \autoref{prop:McKay}, we translate the problem to a computation on $\Flag_d$, where we can do induction using the short exact sequences of \autoref{lem:tautses} below. 
Given $F\in \VB(C)$, we denote the pull-back of the associated tautological bundle along the Flag--Quot morphism by
\[
 F^{\langle d \rangle}:=f_d^*F^{\llb d\rrb}\in \VB(\Flag_d)\,.
\]

\begin{lemma}\label{lem:tautses}
For every $F\in \VB(C)$, we have the following two short exact sequences of vector bundles on $\Flag_d$:
\begin{align}
 &0\to p_1^*F\otimes \reg_p(1)\to  F^{\langle d\rangle}\to p_2^*F^{\langle d-1\rangle}\to 0  \label{eq:tautbundleses}
\\
& 0\to p_2^*F^{\langle d-1\rangle \vee} \to  F^{\langle d\rangle\vee }\to  p_1^*F^\vee\otimes \reg_p(-1) \to 0
\label{eq:tautbundlesesdual}
 \end{align}
\end{lemma}

\begin{proof}
As $f_d\colon \Flag_d\to \Quot_d$ is the classifying morphism for $\wcQ_d$, flat base change gives
\[
 F^{\langle d\rangle}\cong \pr_{2*}\bigl(\pr_1^*F\otimes \wcQ_d\bigr);
\]
compare \cite[Lem.\ 2.4]{Krug--reconstruction}.
We consider the following diagram with a cartesian square and commutative triangles:
\begin{equation}\label{eq:Flagtautdiag}
\xymatrix{
  & \Flag_d\ar^{i=(p_1,\id)}[d] \ar@/^4mm/^{\id}[dr] \ar@/_4mm/_{p_1}[dl]  &   \\
C  & C\times \Flag_d \ar_{\pr_1}[l] \ar^{\pr_2}[r] \ar_{t=\id\times p_2}[d] &  \Flag_d \ar^{p_2}[d]   \\
  & C\times \Flag_{d-1} \ar^{\pr_2}[r]  \ar@/^4mm/^{\pr_1}[ul]   &   \Flag_{d-1}\,.
 }
\end{equation}
We get the sequence \eqref{eq:tautbundleses} by applying the functor $\pr_{2*}(\pr_1^*F\otimes \_)$ to \eqref{eq:sesinduction}. Indeed, the sheaves in \eqref{eq:sesinduction} are finite over $\Flag_d$, so the sequence stays exact. The commutativity of triangles on the upper half of \eqref{eq:Flagtautdiag} together with the projection formula along $i$ give
\[
 \pr_{2*}\bigl(\pr_1^*F\otimes i_*\reg_p(1)\bigr)\cong p_1^*F\otimes \reg_p(1)
\]
which establishes the first term of \eqref{eq:tautbundleses}. Similarly, we get the third term of \eqref{eq:tautbundleses} using flat base change along the cartesian square in \eqref{eq:Flagtautdiag}.
The sequence \eqref{eq:tautbundlesesdual} is just the dual of \eqref{eq:tautbundleses}.
\end{proof}

\begin{lemma}\label{lem:filtration}
 Let $X$ be a variety, and let
 \[
 0\to F_i'\to F_i\to F_i''\to 0
 \]
be short exact sequences of vector bundles for $i=1,\dots m$. Let
$1\le k_i\le \rk F_i$ for $i=1,\dots, m$. Then
$\bigotimes_{i=1}^m \wedge^{k_i}F_i$ has a filtration whose factors are given by
 \[
  \left(\bigotimes_{i=1}^m \wedge^{k_i-\ell_i} F_i' \right)\otimes \left(\bigotimes_{i=1}^m \wedge^{\ell_i} F_i''    \right)\quad\text{ with } 0\le \ell_i\le \min\{k_i, \rk F_i''\} \text{ for } i=1,\dots, m\,.
 \]
\end{lemma}

\begin{proof}
 This is basic linear algebra translated from vector spaces to vector bundles; compare \cite[Exe.\ 5.16]{HarAGbook}.
\end{proof}

\begin{prop}\label{prop:nudualvanish}
Let $F_1,\dots, F_m$ be vector bundles on $C$, and let $1\le k_i\le d\cdot \rk F_i=\rk F_i^{\langle d \rangle}$ such that $\sum_{i=1}^m\min\{k_i, \rk F_i\}<\rk E$. 
Then, for all $d\in \IN$,
\[
 R\nu_{d*}\left(\bigotimes_{i=1}^m\wedge^{k_i} F_i^{\langle d \rangle \vee}\right)\cong 0\,.
\]
\end{prop}

\begin{proof}
We proof the assertion by induction on $d$. For $d=0$, the Quot scheme $\Flag_0=\pt$ is a point with universal quotient $\wcQ_0=0$. Hence, all the $F^{\langle 0\rangle}$ are already $0$, so are their push forwards along $\mu_0\colon \Flag_0=\pt\to C^{(0)}=\pt$. 

The reader who prefers to start the induction at $d=1$, may of course feel free to do so. In this case $\nu_1\colon \Flag_1=\Quot_1=\IP(E)\to C$ is a $\IP^{\rk E-1}$-bundle and $F^{\langle 1\rangle}\cong \nu_1^* F\otimes \reg(1)$, so the proof is still easy.

Let us now do the induction step.
Applying \autoref{lem:filtration} to the short exact sequences of the form \eqref{eq:tautbundlesesdual}, we see that $\bigotimes_{i=1}^m\wedge^{k_i} F_i^{\langle d\rangle\vee}$ has a filtration whose factors are indexed by $\vec \ell=(\ell_1,\dots, \ell_m)$ with $0\le \ell_i\le \min\{k_i,\rk(F_i)\}$ and
given by
\begin{align*}
G_{\vec \ell}&\cong \left(\bigotimes_{i=1}^m\wedge^{k_i-\ell_i} p_2^*F_i^{\langle d-1\rangle \vee}  \right)\otimes \left(\bigotimes_{i=1}^m \wedge^{\ell_i} \bigl(p_1^*F_i^\vee\otimes \reg_p(-1)\bigr)    \right)\\&\cong p^*\left(\Bigl( \bigotimes_{i=1}^m \wedge^{\ell_i} F_i^\vee   \Bigr) \boxtimes \Bigl( \bigotimes_{i=1}^m\wedge^{k_i-\ell_i} F_i^{\langle d-1\rangle \vee}  \Bigr)\right)\otimes \reg_p(-\sum_{i=1}^m \ell_i)\,.
\end{align*}
By assumption $-\rk E\le -\sum_{i=1}^m \ell_i$. Hence, $Rp_* G_{\vec \ell}$ vanishes by the projection formula and \autoref{lem:nupush}(i) for all $\vec \ell\neq \vec 0$. Due to $\nu_d=(\id\times \nu_{d-1})\circ p$, then also
$R\nu_{d*} G_{\vec \ell}\cong 0$.
In the remaining case $\vec\ell=\vec 0$, the push-forward
\[
R\nu_{d*} G_{\vec \ell}\cong R(\id\times \nu_{d-1})_* \Bigl(\reg_C \boxtimes \bigl( \bigotimes_{i=1}^m\wedge^{k_i} F_i^{\langle d-1\rangle \vee}  \bigr) \Bigr)\cong \reg_C\boxtimes\Bigl( R \nu_{d-1*} \bigl( \bigotimes_{i=1}^m\wedge^{k_i} F_i^{\langle d-1\rangle \vee}  \bigr)\Bigr)
\]
vanishes by the induction hypothesis.
As the derived push-forwards of all the factors $G_{\vec \ell}$ vanish, so does the derived push-forward of $\bigotimes_{i=1}^m\wedge^{k_i} F_i^{\langle d\rangle\vee}$.
\end{proof}

\begin{theorem}\label{thm:dualvanish}
Let $F_1,\dots, F_m$ be vector bundles on $C$, and let $1\le k_i\le d\cdot \rk F_i=\rk F_i^{\llb d \rrb}$ such that
$\sum_{i=1}^m\min\{k_i, \rk F_i\}<\rk E$.
Then, for all $d\in \IN$,
\[
 R\mu_*\left(\bigotimes_{i=1}^m\wedge^{k_i} F_i^{\llb d\rrb\vee}\right)\cong 0\,.
\]
\end{theorem}

\begin{proof}
By \autoref{prop:McKay}, we have
\[
\pi^*R\mu_*\left(\bigotimes_{i=1}^m\wedge^{k_i} F_i^{\llb d\rrb\vee}\right)\cong  R\nu_*\left(\bigotimes_{i=1}^m\wedge^{k_i} F_i^{\langle d\rangle}\right)\,,   
\]
which vanishes by \autoref{prop:nudualvanish}. As $\pi$ is faithfully flat (see \autoref{lem:piflat}), it follows that 
$R\mu_*\bigl(\bigotimes_{i=1}^m\wedge^{k_i} F_i^{\llb d\rrb\vee}\bigr)$ already vanishes. 
\end{proof}

\begin{remark}\label{rem:numcondition}
Inspecting the proof of \autoref{prop:nudualvanish}, it looks likely that the numerical condition 
\[
 \sum_{i=1}^m\min\{k_i, \rk F_i\}<\rk E
\]
for the vanishing of the derived push-forward is already sharp.
For example, let $m=1$, and also $d=1$ in which case $\nu_1=\mu_1=p\colon \IP(E)\to C$, and $F^{\llb 1\rrb}\cong p^*F\otimes \reg_p(1)$. Hence, $\wedge^k F^{\llb 1\rrb\vee}\cong p^*(\wedge^k F^\vee)\otimes \reg_p(-k)$. For $\rk F\ge k\ge \rk E$, this gives (see e.g.\ \cite[(A.2b)]{Lazarsfeld--posbookI})
\[
 R\mu_* (\wedge^k F^{\llb 1\rrb\vee})\cong (\wedge^k F^\vee) \otimes (S^{k-\rk E} E) [-\rk E +1]\not\cong 0\,. 
\]
\end{remark}

\section{Fourier--Mukai computations}\label{sect:FM}

\subsection{Fourier--Mukai functors}\label{subsect:FM}

In this section, we will collect some facts on Fourier--Mukai functors that we will need later in \autoref{subsect:Ext}. The standard reference is \cite{Huy}. 

\begin{convention}\label{conv:derived}
From now on, in order to lighten the notation, every functor will tacitly be understood to stand for its derived version on the level of derived categories: $f_*$ stands for $Rf_*$, $f^*$ stands for $Lf^*$, $\otimes$ stands for $\otimes^L$, $\sHom(F,\_)$ stands for $R\sHom(F,\_)$, and the dual $(\_)^\vee$ stands for the derived dual $R\sHom(\_,\reg)$. This convention is quite standard for doing computations with Fourier--Mukai functors. 
\end{convention}

Given two smooth projective varieties $X$ and $Y$, and $\cP\in \D(X\times Y)$, the \emph{Fourier--Mukai transform}
along $\cP$ is
\[
 \FM_{\cP}:=\pr_{2*}\bigl(\pr_1^*(\_)\otimes \cP\bigr)\colon \D(X)\to \D(Y)\,.
\]
The most important examples for us are 
\[
\FM_{\cQ_d}\colon \D(C)\to \D(\Quot_d)\quad,\quad  \FM_{\wcQ_d}\colon \D(C)\to \D(\Flag_d)
\]
where $\cQ_d$ and $\wcQ_d$ are the universal quotient sheaves. For $F\in \VB(C)$, we have
\begin{equation}\label{eq:tautFM}
F^{\llb d\rrb} \cong \FM_{\cQ_d}(F)\quad,\quad F^{\langle d\rangle} \cong \FM_{\wcQ_d}(F)\,.
\end{equation}
Given $\cP$ as above, we also have the Fourier--Mukai transform along by $\cP$ in the reverse direction
\[
 \FM_{\cP}^{Y\to X}:=\pr_{1*}\bigl(\pr_{2*}(\_)\otimes \cP\bigr)\colon \D(Y)\to \D(X)\,.
\]
If we want to emphasise the difference, we denote the Fourier--Mukai transform in the standard direction also by $\FM^{X\to Y}_\cP:=\FM_\cP$.
We set
\begin{equation}\label{eq:PR}
 \cP^R:=\cP^\vee \otimes \pr_1^*\omega_X[\dim X]
\end{equation}

\begin{lemma}\label{lem:RF}
Let $\cP\in \D(X\times Y)$ and $\Phi:=\FM_\cP\colon \D(X)\to \D(Y)$ as above.
\begin{enumerate}
 \item The right-adjoint of $\Phi$ is given by
 \[
  \Phi^R:= \FM_{\cP^R}^{Y\to X}\colon \D(Y)\to \D(X)\,.
 \]
\item Recall that $\pr_{32}\colon X\times Y\times X\to X\times Y$ is the projection to the last two factors, followed by interchanging the factors, which means $\pr_{32}(a,b,c)=(c,b)$. We have  
\[
\Phi^R\circ \Phi\cong \FM_{\cP\star \cP^R}\colon \D(X)\to \D(X)\quad\text{with}\quad \cP\star \cP^R:=\pr_{13}(\pr_{12}^*\cP\otimes \pr_{32}^*\cP)\,.
\]
 
\end{enumerate}
\end{lemma}

\begin{proof}
 Part (i) is \cite[Prop.\ 5.9]{Huy}. Part (ii) is part (i) combined with the fact that compositions of Fourier--Mukai functors are given by Fourier--Mukai functors along the convolution product; see \cite[Prop.\ 5.10]{Huy}.
\end{proof}
It turns out that the Fourier--Mukai functor $\FM_{\cP^R}^{X\to Y}$ along $\cP^R$ in the other direction than $\Phi^R$ is of some use, too. 

\begin{lemma}\label{lem:PRdual}
Let $\Phi=\FM_\cP\colon \D(X)\to \D(Y)$ be a Fourier--Mukai functor. Then, for every $F\in \D(X)$, we have
\[
\FM_{\cP^R}^{X\to Y}(F^\vee)\cong \Phi(F)^\vee\,. 
\]
\end{lemma}

\begin{proof}
 This is an application of the Grothendieck--Verdier duality for  $\pr_2\colon X\times Y\to Y$.
 Indeed, note that \eqref{eq:PR} can be rewritten as 
 $\cP^R:=\cP^\vee \otimes \pr_2^!\reg_Y$.
Hence, 
\begin{align*}
\FM_{\cP^R}^{X\to Y}(F^\vee)\cong \pr_{2*}\bigl(\pr_1^*F^\vee\otimes \cP^\vee \otimes  \pr_2^!\reg_Y  \bigr)  
&\cong \pr_{2*}\sHom(\pr_1^*F\otimes \cP , \pr_2^!\reg_Y  )\\&\cong \sHom\bigl(\pr_{2*}(\pr_1^*F\otimes \cP) , \reg_Y   \bigr)\\&\cong \Phi(F)^\vee\,.\qedhere
\end{align*}
\end{proof}

\begin{lemma}\label{lem:FMhom}
 Let $\Phi=\FM_{\cQ}\colon \D(X_1)\to \D(Y)$ and $\Psi=\FM_{\cP}\colon \D(X_2)\to \D(Y)$ be Fourier--Mukai functors. Then, for every $G\in \D(X_1)$ and $F\in \D(X_2)$, we have
 \[
  \sHom\bigl(\Psi(F), \Phi(G)\bigr)\cong \pr_{2*}\bigl(\pr_1^*G\otimes \pr_{12}^*\cQ\otimes \pr_{32}^*\cP^R\otimes \pr_3^*F^\vee  \bigr)
 \]
where the $\pr$ denote the projections from $X_1\times Y\times X_2$ to its factors.
\end{lemma}

\begin{proof}
We consider the commutative diagram with cartesian square
\[
 \xymatrix{   
X_1\times Y\times X_2\ar_{\delta'\quad}[r] \ar@/^6mm/^{\pr_{32}}[rr] \ar^{\pr_{2}}[d] & X_1\times Y\times Y\times X_2 \ar_{\qquad\pr_{43}}[r] \ar^{\pr_{23}}[d] & X_2\times Y  \\
Y \ar^{\delta}[r]& Y\times Y  & 
 }
\]
where $\delta\colon Y\hookrightarrow Y\times Y$ is the diagonal embedding, and $\delta'=\id\times \delta\times \id$.

First, the compatibility of Fourier--Mukai functors with box products (see \cite[Exe.\ 5.13(i)]{Huy}) together with \autoref{lem:PRdual} gives in $\D(Y\times Y)$ the isomorphism 
\[
\Phi(G)\boxtimes\Bigl(\Psi(F)^\vee\Bigr)\cong \pr_{23*}\bigl(\pr_1^*G\otimes \pr_{12}^*\cQ\otimes \pr_{43}^*\cP^R\otimes \pr_4^*F^\vee  \bigr)\,. 
\]
Then, using flat base change along the above cartesian square, we get
\begin{align*}
\sHom\bigl(\Psi(F),\Phi(G)\bigr) &\cong \delta^*\left(\Phi(G)\boxtimes\Bigl(\Psi(F)^\vee\Bigr)\right)\\&\cong \delta^*\pr_{23*}\bigl(\pr_1^*G\otimes \pr_{12}^*\cQ\otimes \pr_{43}^*\cP^R\otimes \pr_4^*F^\vee  \bigr) 
\\&\cong \pr_{2*}\delta'^*\bigl(\pr_1^*G\otimes \pr_{12}^*\cQ\otimes \pr_{43}^*\cP^R\otimes \pr_4^*F^\vee  \bigr)
\\
&\cong \pr_{2*}\bigl(\pr_1^*G\otimes \pr_{12}^*\cQ\otimes \pr_{32}^*\cP^R\otimes \pr_3^*F^\vee  \bigr)\,. \qedhere
\end{align*}
\end{proof}

\autoref{lem:FMhom} together with \eqref{eq:tautFM} gives the following description of the Hom bundle between two tautological bundles, which we will use later in the proof of \autoref{thm:mainExt}.

\begin{cor}\label{cor:HomtautFM}
For every $d\in \IN$, and all $F,G\in \VB(C)$, we have
\[
 \sHom(F^{\llb d\rrb}, G^{\llb d\rrb})\cong \pr_{2*}\Bigl(\pr_1^*G\otimes \pr_{12}^*\cQ_d\otimes \pr_{32}^*\cQ_d^R\otimes \pr_3^* F^\vee  \Bigr)\,.  
\] 
\end{cor}

\subsection{Equivariant sheaves}\label{subsect:equi}

We only collect the few facts about equivariant sheaves, categories, and functors that we need in the following. For any further details, including full definitions of the basic notions, we refer to \cite[Sect.\ 4]{BKR}, \cite{Elagin}, \cite[Sect.\ 2.2]{Krug--remarksMcKay}, or \cite[Sect.\ 2.5]{Krug--curveExt},  to name a few of the possible references.

Let $G$ be a finite group acting on a smooth projective variety $X$. A \emph{$G$-linearisation} on a sheaf $\cF\in \Coh(X)$ is a collection of isomorphisms $\{\lambda_g\colon \cF\xrightarrow\sim g^*\cF\}_{g\in G}$ satisfying a cocycle condition. The sheaf $\cF$ equipped with a $G$-linearisation $\lambda$ is called a \emph{$G$-equivariant sheaf}. There is an abelian category $\Coh_G(X)$ of $G$-equivariant sheaves, and we denote its bounded derived category by $\D_G(X):=\D^b(\Coh_G(X))$.
Given a $G$-equivariant morphism $f\colon X\to Y$, we get induced equivariant push-forward and pull-back functors
\begin{align*}
 f_*\colon \Coh_G(X)\to \Coh_G(Y)\quad&,\quad f_*\colon \D_G(X)\to \D_G(Y)\,,\\ f^*\colon \Coh_G(Y)\to \Coh_G(X)\quad&,\quad f^*\colon \D_G(Y)\to \D_G(X)\,.
\end{align*}
If $G$ acts trivially on $Y$, a $G$-linearisation of $\cF\in \Coh(Y)$ is just a $G$-action $G\to \Aut(\cF)$. Hence, we have the functor $\triv\colon \Coh(Y)\to \Coh_G(Y)$, equipping every sheaf with the trivial $G$-action, and the functor $(\_)^G\colon \Coh_G(Y)\to \Coh(Y)$ sending every sheaf to its subsheaf of $G$-invariants. Both functors are exact, and we denote the induced functors on the derived categories again by $\triv\colon \D(Y)\to \D_G(Y)$ and $(\_)^G\colon \D_G(Y)\to \D(Y)$.
If $f\colon X\to Y$ is a $G$-invariant morphism (i.e.\ it is $G$-equivariant when we equip $Y$ with the trivial $G$-action), then we write
\[
 f^G_*:=(\_)^G\circ f_*\colon \D_G(X)\to \D(Y) \quad,\quad f_G^*:= f^*\circ \triv\colon \D(Y)\to \D_G(X)\,.
\]

\begin{lemma}\label{lem:pifaithful}
Let a finite group $G$ act on a smooth variety $X$.
Let $\pi\colon X \to Y=X/G$ be the geometric quotient, and assume that $Y$ is again smooth. 
Then we have an isomorphism
\[
 \pi_*^G\circ \pi^*_G\cong \id_{\D(Y)}\,,
\]
which means that $\pi^*_G\colon \D(Y)\to \D_G(X)$ is fully faithful.
\end{lemma}

\begin{proof}
 This is the equivariant projection formula together with the fact that the geometric quotient satisfies $\pi_*^G\reg_X\cong \reg_Y$; see e.g.\ \cite[Lem.\ 2.1]{Krug--remarksMcKay} for details.
\end{proof}

We get the following criterion for an equivariant sheaf to descent to the quotient.

\begin{cor}\label{cor:descent}
An object $\cF\in \D_G(X)$ is of the form $\pi^*_G\cE$ for some $\cE\in \D(Y)$ if and only if 
$\cF\cong \pi^*_G\pi_*^G\cF$.
\end{cor}

Let $d\ge 2$, let the symmetric group $\sym_d$ act on a smooth projective variety $X$, and let $Z\subset X$ be a $\sym_d$-invariant closed subscheme. Then there is a \emph{canonical} $\sym_d$-linearisation of $\reg_Z$ given by push-forward of regular functions along the group action:
 \[
  \lambda_g:=(g^\#)^{-1}\colon \reg_Z\to g^*\reg_Z\quad\text{for $g\in \sym_d$.}
 \]
In addition, there is the \emph{anti-canonical} linearisation of $\reg_Z$ given by 
 \[
  \overline \lambda_g:=\sgn(g)\lambda_g\colon \reg_Z\to g^*\reg_Z\quad\text{for $g\in \sym_d$.}
 \]

\begin{lemma}\label{lem:twolin}
Let $Z\subset X$ be a $\sym_d$-invariant closed subscheme. 

\begin{enumerate}
 \item If $Z$ is reduced and connected, then the canonical and the anti-canonical linearisation are the only two $\sym_d$-linearisations of $\reg_Z$. 
\item Let $\pi\colon X\to Y$ be a $\sym_d$-invariant morphism.
Assume that there is some point $z\in Z$ whose stabiliser $G_z$ contains some $g_z\in \sym_d$ with  
$\sgn(g_z)=-1$. Then $\reg_Z$ equipped with the anti-canonical linearisation does not descent along $\pi$, which means that there is no $\cE\in \D(Y)$ with $\pi_{\sym_d}^*\cE\cong (\reg_Z,\overline\lambda)$.
 \end{enumerate}
\end{lemma}

\begin{proof}
If $Z$ is connected and reduced, then $\End(\reg_Z)\cong \IC$. Hence, by \cite[Lem.\ 1]{Plo}, the set of linearisations of $\reg_Z$ is a $\widehat \sym_d$-torsor, where $\widehat \sym_d\cong \IZ/2\IZ$ is the group of characters of $\sym_d$.

For part (ii), assume for a contradiction that $(\reg_Z,\lambda)\cong \pi^*_{\sym_d} \cE$. Let $i_z\colon\{z\}\hookrightarrow X$ be the inclusion of the point, which is a $G_z$-equivariant morphism. Then, $(i_z\circ \pi)^*\cE\in \D_{G_z}(\{z\})$ is equipped with the trivial $G_z$ action, as $\cE$ is. On the other hand, the element $g_z$ acts on $\reg_z\cong \cH^0\bigl(i_z^*(\reg_Z,\overline \lambda)\bigr)\cong \cH^0\bigl((i_z\circ \pi)^*\cE\bigr)$ by $-1$.
\end{proof}

\subsection{A few subschemes related to the universal divisor $\Xi$}\label{subsect:Xi}

Recall that the universal family of effective degree $d$ divisors on $C$ is the reduced subscheme
\[
 \Xi_d=\bigl\{(x;x_1+\dots+x_d)\mid x\in \{x_1,\dots, x_d\}\bigr\}\subset C\times C^{(d)}\,.
\]
Equivalently, it is the image of the embedding \eqref{eq:Xiembedding}.
We define $D_d:=(\id\times \pi)^{-1}\Xi_d\subset C\times C^d$ as the scheme-theoretic pre-image under the quotient morphism $\id\times \pi\colon C\times C^d\to C\times C^{(d)}$. It turns out that $D_d$ is the reduced subscheme given by 
\[
 D_d=\bigl\{(x;x_1,\dots,x_d)\mid x\in \{x_1,\dots, x_d\}\bigr\}\subset C\times C^{d}\,;
\]
see \cite[Lem.\ 3.2]{Krug--curveExt}
We consider $\sym_d$ acting on $C\times C^d$ by the identity on the first factor $C$ and by permutation on the factor $C^d$. 
We see that the points of $\Xi_d$ correspond to the $\sym_d$-orbits of $D_d$. Furthermore, by the definition of $D_d$ as the scheme-theoretic preimage, we have $(\id\times \pi)^*\reg_{\Xi_d}\cong \reg_{D_d}$ as objects in $\D_{\sym_d}(C\times C^d)$. Hence, by \autoref{lem:pifaithful}, $(\id\times \pi)^{\sym_d}_*\reg_{\D_d}\cong \reg_{\Xi_d}$. In summary, $\Xi_d\cong D_d/\sym_d$ is the geometric quotient.

In \autoref{subsect:Ext}, we will consider the images of $D_d$ and $\Xi_d$ under diagonal embeddings
\begin{align*}
 \iota_d\colon C\times C^d\to C\times C^d\times C \quad &,\quad (x;x_1,\dots,x_n)\mapsto (x;x_1,\dots,x_n;x)\,,\\
 \iota_{(d)}\colon C\times C^{(d)}\to C\times C^{(d)}\times C \quad &,\quad (x;x_1+\dots+x_n)\mapsto (x;x_1+\dots+x_n;x)\,,
\end{align*}
which are the reduced subschemes
\begin{align*}
\Gamma_d:&=\iota_d(D_d)= \bigl\{(x;x_1,\dots,x_d;x)\mid x\in \{x_1,\dots, x_d\}\bigr\}\subset C\times C^{d}\times C\,,\\
\Gamma_{(d)}:&=\iota_{(d)}(\Xi_d)= \bigl\{(x;x_1+\dots+x_d;x)\mid x\in \{x_1,\dots, x_d\}\bigr\}\subset C\times C^{(d)}\times C\,. 
\end{align*}
Again, $\Gamma_{(d)}\cong \Gamma_d/\sym_d$ is the geometric quotient, where we consider $\sym_d$ acting by permutation of the inner factor $C^d$
of $C\times C^d\times C$. In particular,
\begin{equation}\label{eq:Gammainva}
 (\id\times \pi\times \id)_*^{\sym_d}\reg_{\Gamma_d}\cong \reg_{\Gamma_{(d)}}
\end{equation}
where $\id\times \pi\times \id\colon C\times C^d\times C \to C\times C^{(d)}\times C$ is the quotient morphism.

\subsection{Computation of the Extension Groups}\label{subsect:Ext}

In this subsection, we will prove \autoref{thm:mainExt} and \autoref{thm:mainfunctor}. As this is probably the most difficult part of this paper, let us give a short summary of what will happen.

Recall the notation of \autoref{sect:McKay}, in particular \eqref{eq:Flagdiag}.
We will use computations involving the Fourier--Mukai kernels $\cQ_d$ and $\cQ_d^R$, and their pull-backs $\wcQ_d$ and $\wcQ_d^R$ along the Flag--Quot morphism $f_d\colon \Flag_d\to \Quot_d$. 
We set
\[
 \alpha_d:=\id\times \nu_d\times \id\colon C\times \Flag_d\times C\to C\times C^d\times C\,.
\]
\autoref{prop:McKay} tells us that
\begin{equation}\label{eq:alphaconv}
\alpha_{d*}\bigl(\pr_{12}^*\wcQ_d\otimes \pr_{32}^*\wcQ_d^R\bigr) 
\end{equation}
contains much information about 
\begin{equation}\label{eq:muconv}
(\id\times \mu_d\times \id)_*\bigl(\pr_{12}^*\cQ_d\otimes \pr_{32}^*\cQ_d^R\bigr); 
\end{equation}
see \autoref{lem:convinva} for an exact statement about the relationship of
\eqref{eq:alphaconv} and \eqref{eq:muconv}. Hence, the first step is to provide a simple description of \eqref{eq:alphaconv}, which is what we do in
\autoref{prop:FMpush}. The proof of \autoref{prop:FMpush} by induction on $d$ will go over several pages and include the auxiliary \autoref{lem:adjointtriangle}, \autoref{lem:QRvanish}, \autoref{lem:maincomp},  \autoref{eq:unionses}, and \autoref{lem:RFsect}. 

Once the description for \eqref{eq:alphaconv} is established, we  use it to achieve a similar description for 
\eqref{eq:muconv} in \autoref{prop:FMSdC}. After we computed \eqref{eq:muconv}, it is quite easy to deduce \autoref{thm:mainExt} (see \autoref{thm:muHom}) using \autoref{cor:HomtautFM}, and to deduce \autoref{thm:mainfunctor} (see \autoref{thm:convolution}) using \autoref{lem:RF}. 

\begin{lemma}\label{lem:convinva}
The object $\alpha_{d*}\bigl(\pr_{12}^*\wcQ\otimes \pr_{32}^*\wcQ^R\bigr)\in \D(C\times C^d\times C)$ carries a $\sym_d$-linearisation such that, in $\D_{\sym_d}(C\times C^{d}\times C)$ and in $\D(C\times C^{(d)}\times C)$, respectively, we have isomorphisms
\begin{align}
&(\id\times \pi_d\times \id)^*_{\sym_d}(\id\times \mu_d \times \id)_*\bigl(\pr_{12}^*\cQ\otimes \pr_{32}^*\cQ^R\bigr)\cong \alpha_{d*}\bigl(\pr_{12}^*\wcQ\otimes \pr_{32}^*\wcQ^R\bigr)\,,\label{eq:FMpull}  \\
&(\id\times \mu_d \times \id)_*\bigl(\pr_{12}^*\cQ\otimes \pr_{32}^*\cQ^R\bigr)\cong (\id\times \pi_d\times \id)_*^{\sym_d}\alpha_{d*}\bigl(\pr_{12}^*\wcQ\otimes \pr_{32}^*\wcQ^R\bigr)
\label{eq:FMinva}\,.
\end{align}
\end{lemma}

\begin{proof}
 Recall that $\alpha_d=\id\times \nu_d\times \id$. Hence, \autoref{prop:McKay} gives
\begin{equation}\label{eq:McKaybasechange}
(\id\times \pi_d\times \id)^*(\id\times \mu_d \times \id)_* \cong \alpha_{d*}(\id\times f_d\times \id)^*\,.
\end{equation}
Since $f_d\colon \Flag_d\to \Quot_d$ is the classifying morphism for $\wcQ_d$, we get $(\id\times f_d)^*\cQ_d\cong \wcQ_d$. As duals commute with pull-backs, and $(\id\times f_d)^*(\omega_C\boxtimes\reg_{\Quot_d})\cong \omega_C \boxtimes\reg_{\Flag_{d}}$, we also have 
$(\id\times f_d)^*(\cQ_d^R)\cong \wcQ_d^R$; see \eqref{eq:PR}. Hence, \eqref{eq:McKaybasechange} gives the isomorphism \eqref{eq:FMpull}.
Finally, composing both sides of \eqref{eq:FMpull} with $(\id\times \pi_d\times \id)_*^{\sym_d}$ and applying \autoref{lem:pifaithful} gives \eqref{eq:FMinva}.
\end{proof}

\begin{prop}\label{prop:FMpush} For every $d\in \IN$, we have
\[
 \alpha_{d*}\bigl(\pr_{12}^*\wcQ_d\otimes \pr_{32}^*\wcQ_d^R\bigr)\cong \reg_{\Gamma_d}\,.
\]
\end{prop}

We will do the proof of \autoref{prop:FMpush} by induction on $d$. Setting $\Gamma_0:=\emptyset\subset C\times C$, we can consider $d=0$ as the initial case, where the proof is trivial as $\wcQ_0=0$. Alternatively, the reader may do $d=1$ as the initial case, in which case the proof is not trivial, but easy.

As part of the induction step, we have to proof five technical lemmas. Note that in the proof of the third  
\autoref{lem:maincomp}, we will already use the induction hypothesis that \autoref{prop:FMpush} holds for $d-1$.

\begin{lemma}\label{lem:adjointtriangle}
In $\D(C\times \Flag_d)$, there is an exact triangle 
\begin{equation}\label{eq:FMinductionses}
t^*(\wcQ_{d-1}^R)\to \wcQ_d^R\to i_*\reg_p(-1)\to t^*(\wcQ_{d-1}^R)[1]\,.
\end{equation}
\end{lemma}
\begin{proof}
We get \eqref{eq:FMinductionses} by applying the contravariant exact functor $(\_)^R=(\_)^\vee\otimes(\omega_C\boxtimes \reg_{\Flag_d})[1]$ to \eqref{eq:sesinduction}; compare \eqref{eq:PR}. The middle term of the resulting exact triangle is $\wcQ_d^R$. 

In order to show that the third term of the resulting triangle becomes $i_*\reg_p(1)^R\cong i_*\reg_p(-1)$, note that $i\colon \Flag_d\hookrightarrow C\times \Flag_d$ is the closed embedding of a divisor with normal bundle $N_i\cong \omega_{\Flag_d}\otimes i^*\omega_{C\times \Flag_d}^\vee\cong i^*(\omega_C^\vee\boxtimes \reg_{\Flag_d})$. By Grothendieck--Verdier duality along $i$, we get
\[
(i_*\reg_p(1))^\vee\cong i_*\bigl(\reg_p(-1)\otimes i^*(\omega_C^\vee\boxtimes \reg_{\Flag_d})\bigr)[-1]\,.  
\] 
It follows that 
\[
  i_*\reg_p(1)^R\cong (i_*\reg_p(1))^\vee\otimes (\omega_C\boxtimes\reg_{\Flag_d})[1]\cong i_*\reg_p(-1)\,.
\]
The isomorphism $(t^*\wcQ_{d-1})^R\cong t^*(\wcQ_{d-1}^R)$ for the first term follows from the facts that duals commute with pull-backs, and $t^*(\omega_C\boxtimes\reg_{\Flag_{d-1}})\cong \omega_C \boxtimes\reg_{\Flag_{d}}$.
\end{proof}

\begin{lemma}\label{lem:QRvanish}
 For all $d\in \IN$, we have $(\id\times \nu_d)_*  \wcQ_d^R\cong 0$.
\end{lemma}

\begin{proof}
We consider the commutative diagram 
\begin{equation}\label{eq:easycomm}
 \xymatrix{
C\times \Flag_d\ar^{\pr_2}[r] \ar_{\id\times \nu_d}[d] &   \Flag_d \ar^{\nu_d}[d]   \\
C\times C^d \ar^{\pr_2}[r] &    C^d\,.
 }
\end{equation}
Recall from \eqref{eq:tautFM} that $F^{\langle d\rangle}\cong \FM_{\wcQ_d}(F)$ for $F\in \VB(C)$. 
Thus, by \autoref{lem:PRdual}, we have $\pr_{2*}\wcQ_d^R\cong \FM_{\wcQ_d^R}(\reg_C^\vee)\cong \reg_C^{\langle d\rangle \vee}$. Hence, the commutativity of \eqref{eq:easycomm} together with the $m=1$ case of the vanishing result \autoref{prop:nudualvanish} give
\[
 \pr_{2*}(\id\times\nu_d)_* \wcQ_d^R \cong  \nu_{d*}\pr_{2*}\wcQ_d^R \cong \nu_{d*} \reg_C^{\langle d\rangle \vee}\cong 0\,.
\]
As we know a priori that $\supp\bigl( (\id\times\nu_d)_* \wcQ_d^R\bigr)\subset (\id\times\nu_d)\bigl(\supp \wcQ^R\bigr)=D_d$ is finite over $C^d$, the vanishing of $\pr_{2*}(\id\times\nu_d)_* \wcQ_d^R$ implies the asserted vanishing of $(\id\times\nu_d)_* \wcQ_d^R$. 

Alternatively, one can also proof the assertion by induction on $d$ using \eqref{eq:FMinductionses}. 
\end{proof}

We want to use \eqref{eq:FMinductionses} for our induction step of the proof of \autoref{prop:FMpush}. To lighten the notation,  we set
\[
 [\cF,\cG]:=\alpha_{d*}\bigl(\pr_{12}^*\cF\otimes \pr_{32}^*\cG\bigr)\in \D(C\times C^d\times C)
\]
for $\cF,\cG\in \D(C\times \Flag_d)$.
We abbreviate the object that we want to compute even further as 
\[\cS:=[\wcQ_d,\wcQ_{d}^R]= \alpha_{d*}\bigl(\pr_{12}^*\wcQ_d\otimes \pr_{32}^*\wcQ_d^R\bigr)\,.\]
From the short exact sequence \eqref{eq:sesinduction} and the exact triangle \eqref{eq:FMinductionses},  
we have the diagram 
\begin{equation}\label{eq:convodiagbefore}
\xymatrix{
[i_*\reg_p(1),t^*\wcQ_{d-1}^R] \ar[d] \ar[r]  &  [i_*\reg_p(1),\wcQ_d^R]  \ar[r] \ar[d]   &  [i_*\reg_p(1),i_*\reg_p(-1)]  \ar[d]    \\
 [\wcQ_d,t^*\wcQ_{d-1}^R]  \ar[d]  \ar[r]  &  \cS \ar[d]  \ar[r]  &  [\wcQ_d,i_*\reg_p(-1)] \ar[d]   \\
[t^*\wcQ_{d-1}, t^*\wcQ_{d-1}^R]   \ar[r]   & [t^*\wcQ_{d-1},\wcQ_{d}^R]   \ar[r]   &  [t^*\wcQ_{d-1},i_*\reg_p(-1)] 
}
\end{equation}
where all rows and columns are exact triangles in $\D(C\times C^d\times C)$. We now compute the corners of this diagram.

\begin{lemma}\label{lem:maincomp}
For $d\ge 1$, we consider the reduced closed subschemes of $C\times C^d\times C$ given by
\begin{align*}
 T&=\bigl\{(x;x_1,x_2,\dots,x_n;x)\mid x\in\{x_2,\dots,x_n\}\bigr\} \,,\\ 
A&=\bigl\{(x;x,x_2,\dots,x_n;x)\bigr\}\,,\\
B&=A\cap T=\bigl\{(x;x,x_2,\dots,x_n;x)\mid x\in\{x_2,\dots,x_n\}\bigr\}\,.
 \end{align*}
Note that $T=B=\emptyset$ for $d=1$. We have isomorphisms
\begin{align*}
[t^*\wcQ_{d-1}, t^*\wcQ_{d-1}^R]&:=\alpha_{d*}\bigl((\pr_{12}^*t^*\wcQ_{d-1})\otimes (\pr_{32}^*t^*\wcQ_{d-1}^R)\bigr)\cong \reg_T\,,\tag{i}\\
[t^*\wcQ_{d-1},i_*\reg_p(-1)]&:=\alpha_{d*}\bigl((\pr_{12}^*t^*\wcQ_{d-1})\otimes (\pr_{32}^*i_*\reg_p(-1))\bigr)\cong 0\,,\tag{ii}\\
[i_*\reg_p(1),i_*\reg_p(-1)]&:=\alpha_{d*}\bigl((\pr_{12}^*i_*\reg_p(1))\otimes (\pr_{32}^*i_*\reg_p(-1))\bigr)\cong \reg_{A}\,, \tag{iii}\\
[i_*\reg_p(1),t^*\wcQ_{d-1}^R]&:=\alpha_{d*}\bigl((\pr_{12}^*i_*\reg_p(1))\otimes (\pr_{32}^*t^*\wcQ_{d-1}^R)\bigr)\cong \reg_{B}[-1]\,.\tag{iv}\\
\end{align*}
\end{lemma}

\begin{proof}
We will use the diagram 
\begin{equation}\label{eq:bigdiag}
  \xymatrix{
  &   & \Flag_d \ar@/_6mm/_{\id}[dll] \ar@/^6mm/^{\id}[drr] \ar_{(\id,p_1)}[dl] \ar_{(p_1,\id,p_1)}[dd]  \ar^{i}[dr]  &   &   \\
  \Flag_d \ar_{i=(p_1,\id)}[dr] & \Flag_d\times C \ar^{i\times \id}[dr] \ar_{\pr_1}[l]  &   & C\times\Flag_d \ar_{\id\times(\id,p_1)}[dl] \ar^{\pr_2}[r]  & \Flag_d \ar^{i=(p_1,\id)}[dl]  \\
 &C\times \Flag_d \ar_{t=\id\times p_2}[d] & C\times \Flag_d\times C \ar_{\pr_{12}}[l]  \ar^{\pr_{32}}[r]  \ar_{\id\times p\times \id}[d]  & C\times \Flag_d \ar^{t=\id\times p_2}[d]&\\
 &C\times \Flag_{d-1} & C\times C\times \Flag_{d-1}\times C \ar_{\pr_{13}\quad}[l]  \ar^{\quad\pr_{43}}[r]  \ar^{\pr_{134}}[d] & C\times \Flag_{d-1}&\\
 & & C\times \Flag_{d-1}\times C \ar^{\pr_{12}}[ul]  \ar_{\pr_{32}}[ur] &  & 
  }
 \end{equation}
where the two squares and the two parallelograms are cartesian, and all triangles are commutative.
We also set 
\begin{align*}\alpha_{d-1}':=\id\times\alpha_{d-1}=\id\times\id\times\nu_{d-1}\times \id\colon C\times C\times \Flag_{d-1}\times C\to& C\times C\times C^{d-1}\times C\\&=C\times C^d\times C\end{align*}
which gives the factorisation
\begin{equation}\label{eq:alphafactor}
 \alpha_d=\alpha_{d-1}'\circ (\id\times p\times \id)\,;
\end{equation}
compare \eqref{eq:nufactor}. 
For part (i), we use the lower part of diagram \eqref{eq:bigdiag} to get
\begin{align*}
&\alpha_{d*}\bigl((\pr_{12}^*t^*\wcQ_{d-1})\otimes (\pr_{32}^*t^*\wcQ_{d-1}^R)\bigr)\\\cong &\alpha_{d-1*}'(\id\times p\times \id)_*\bigl((\pr_{12}^*t^*\wcQ_{d-1})\otimes (\pr_{32}^*t^*\wcQ_{d-1}^R)\bigr) \tag{by \eqref{eq:alphafactor}}
\\\cong &\alpha_{d-1*}'(\id\times p\times \id)_*(\id\times p\times \id)^*\pr_{134}^*\bigl((\pr_{12}^*\wcQ_{d-1})\otimes (\pr_{32}^*\wcQ_{d-1}^R)\bigr)\tag{commutativity of \eqref{eq:bigdiag}}\\
\cong &\alpha_{d-1*}'\pr_{134}^*\bigl((\pr_{12}^*\wcQ_{d-1})\otimes (\pr_{32}^*\wcQ_{d-1}^R)\bigr)\tag{by \autoref{lem:nupush}(iv)}
\end{align*} 
Applying flat base change along the diagram
\[
 \xymatrix{
 &C\times C\times \Flag_{d-1}\times C \ar_{\alpha_{d-1}'}[d] \ar^{\quad\pr_{134}}[r]&  C\times \Flag_{d-1}\times C \ar^{\alpha_{d-1}}[d]\\
C\times C^d\times C \ar@{=}[r]& C\times C\times C^{d-1}\times C \ar^{\pr_{134}}[r]  & C\times C^{d-1}\times C
 }
\]
and the induction hypothesis that the assertion of \autoref{prop:FMpush} is true for $d-1$, we get 
\begin{align*}
\alpha_{d-1*}'\pr_{134}^*\bigl((\pr_{12}^*\wcQ_{d-1})\otimes (\pr_{32}^*\wcQ_{d-1}^R)\bigr)
 \cong &\pr_{134}^*\alpha_{d-1*}\bigl((\pr_{12}^*\wcQ_{d-1})\otimes (\pr_{32}^*\wcQ_{d-1}^R)\bigr)\\
 \cong &\pr_{134}^*\reg_{\Gamma_{d-1}}\\
 \cong &\reg_T\,.
\end{align*}
For part (ii), commutativity on the lower left of \eqref{eq:bigdiag} together with flat base change along the upper right parallelogram of \eqref{eq:bigdiag}, followed by the projection formula for $\id\times p\times \id$ gives 
\begin{align}
&\alpha_{d*}\Bigl(\bigl((\pr_{12}^*t^*\wcQ_{d-1}\bigr)\otimes \bigl(\pr_{32}^*i_*\reg_p(-1)\bigr)\Bigr)\notag\\
 \cong &\alpha_{d*}\Bigl(\bigl((\id\times p\times \id)^*\pr_{13}^*\wcQ_{d-1}\bigr)\otimes \bigl((\id\times (\id,p_1))_*(\reg_C\boxtimes \reg_p(-1))\bigr)\Bigr)\notag
\\
 \cong &\alpha'_{d-1*}\Bigl(\bigl(\pr_{13}^*\wcQ_{d-1}\bigr)\otimes (\id\times p\times \id)_*(\id\times (\id,p_1))_*(\reg_C\boxtimes \reg_p(-1))\Bigr) \label{eq:vanishing}
\end{align}
Now, note that we have a commutative diagram
\[
 \xymatrix{
C\times \Flag_d \ar^{\id\times p\quad}[r] \ar_{\id\times (\id,p_1)}[d] & C\times C\times \Flag_{d-1}  \ar^j[d]  \\
C\times \Flag_d\times C \ar^{\id\times p\times \id\quad}[r] & \quad C\times C\times \Flag_{d-1}\times C
 }
\]
where $j(x,y,z)=(x,y,z,y)$. We have $p_*\reg(-1)=0$; see \autoref{lem:nupush}(i) (here we use the general assumption that $\rk E\ge 2$). Hence we get the vanishing of  
\[
 (\id\times p\times \id)_*(\id\times (\id,p_1))_*(\reg_C\boxtimes \reg_p(-1))\cong j_*(\id\times p)_*(\reg_C\boxtimes \reg_p(-1)))\,,
\]
and accordingly the desired vanishing of \eqref{eq:vanishing}.

In order to prove part (iii), we first apply flat base change along the parallelograms of \eqref{eq:bigdiag} to both tensor factors, which gives
\begin{align*}
[i_*\reg_p(1),i_*\reg_p(-1)]\cong \alpha_{d*}\Bigl(\bigl((i\times \id)_*\pr_1^*\reg_p(1)\bigr)\otimes \bigl((\id\times ( \id,p_1))_*\pr_2^*\reg_p(-1)\bigr)\Bigr)\,.
\end{align*}
Now, note that the rhombus on top of \eqref{eq:bigdiag} is cartesian too, but all of the involved morphism are closed embeddings so none of them is flat. However, as these embeddings are regular, and the intersection is transversal, the base change along this diagram is still exact; see \cite[Cor.\ 2.27]{Kuz}. Using this, together with the projection formula (first along $i\times \id$, then along $(\id, p_1)$) and the commutativity of the curved triangles on top of \eqref{eq:bigdiag}, gives
\begin{align*}
\alpha_{d*}\Bigl(\bigl((i\times \id)_*\pr_1^*\reg_p(1)\bigr)\otimes \bigl((\id\times ( \id,p_1))_*\pr_2^*\reg_p(-1)\bigr)\Bigr)
&\cong \alpha_{d*}(p_1,\id,p_1)_*\bigr(\reg_p(1)\otimes \reg_p(-1)\bigl)\\&\cong \alpha_{d*}(p_1,\id,p_1)_*\reg_{\Flag_d}\,.
\end{align*}
To proceed, we note that there is a commutative diagram
\[
 \xymatrix{
 \Flag_d \ar^{\nu_d}[r] \ar_{(p_1,\id,p_1)}[d] & C^d\ar@{=}[r] & C\times C^{d-1}  \ar^\iota[d]  \\
C\times \Flag_d\times C \ar^{\alpha_d}[r] &  C\times C^d\times C \ar@{=}[r] & C\times C\times C^{d-1}\times C
 }
\]
where $\iota(x,x_2,\dots,x_d)=(x;x,x_2\dots,x_n;x)$.
Hence, by \autoref{lem:nupush}(v), we get
\[
 \alpha_{d*}(p_1,\id,p_1)_*\reg_{\Flag_d}\cong \iota_*\nu_{d*}\reg_{\Flag_d}\cong\iota_*\reg_{C^d}\cong \reg_A\,.
\]
The start of the proof of part (iv) is similar to that of (ii). Commutativity on the lower right of \eqref{eq:bigdiag} together with flat base change along the upper left parallelogram of \eqref{eq:bigdiag} followed by the projection formula along $\id\times p\times \id$ gives 
\begin{align*}
[i_*\reg_p(1)),t^*\wcQ_{d-1}^R]
 \cong \alpha'_{d-1*}\Bigl( \bigl((\id\times p\times \id)_*(i\times\id)_*(\reg_p(1)\boxtimes \reg_C)\bigr)\otimes (\pr_{43}^*\wcQ_{d-1}^R) \Bigr)\,. 
\end{align*}
Denoting by $\delta\colon C\to C\times C$ the diagonal embedding, we have a commutative diagram 
\[
 \xymatrix{
 \Flag_d\times C \ar^{p\times \id}[r] \ar_{i\times\id}[d] & C\times \Flag_{d-1}\times C\ar^{\alpha_{d-1}}[r]\ar^{\delta\times \id\times\id}[d] & C\times C^{d-1}\times C  \ar^{\delta\times \id\times\id}[d]  \\
C\times \Flag_d\times C \ar^{\id\times p\times \id\quad}[r] & \quad C\times C\times \Flag_{d-1}\times C \ar^{\quad\alpha'_{d-1}}[r] & C\times C\times C^{d-1}\times C
 }
\]
and also the equality $\pr_{43}\circ(\delta\times \id\times \id)=\pr_{32}$ of morphisms $C\times \Flag_{d-1}\times C\to C\times \Flag_{d-1}$. Hence, using also \autoref{lem:nupush}(ii), we have
\begin{align}
 & \alpha'_{d-1*}\Bigl( \bigl((\id\times p\times \id)_*(i\times\id)_*(\reg_p(1)\boxtimes \reg_C)\bigr)\otimes (\pr_{43}^*\cQ_{d-1}^R) \Bigr) \notag   \\
 \cong & (\delta\times \id\times \id)_* \alpha_{d-1*}\Bigl( \bigl((p\times \id)_*(\reg_p(1)\boxtimes \reg_C)\bigr)\otimes (\pr_{32}^*\cQ_{d-1}^R) \Bigr)\notag\\
\cong & (\delta\times \id\times \id)_* \alpha_{d-1*}\Bigl( \pr_{12}^*\wcK_{d-1}\otimes \pr_{32}^*\cQ_{d-1}^R \Bigr)\,.\label{eq:part4}
 \end{align}
From the $d-1$ case of \autoref{lem:QRvanish}, we get the vanishing $\alpha_{d-1*}\bigl(\pr_1^*E\otimes \pr_{32}^*\wcQ_{d-1}^R\bigr) \cong 0$.
Hence, the short exact sequence
\[
 0\to \pr_{12}^*\wcK_{d-1} \to \pr_1^*E\to \pr_{12}^*\wcQ_{d-1}\to 0
\]
on $C\times \Flag_{d-1}\times C$ together with the induction hypothesis for \autoref{prop:FMpush} give
\[
\alpha_{d-1*}\Bigl( \pr_{12}^*\wcK_{d-1}\otimes \pr_{32}^*\cQ_{d-1}^R \Bigr)\cong \alpha_{d-1*}\Bigl( \pr_{12}^*\wcQ_{d-1}\otimes \pr_{32}^*\cQ_{d-1}^R \Bigr)[-1]\cong \reg_{\Gamma_{d-1}}[-1]\,.
\]
Plugging this into \eqref{eq:part4} gives the asserted
\[
(\delta\times \id\times \id)_* \alpha_{d-1*}\Bigl( \pr_{12}^*\wcK_{d-1}\otimes (\pr_{32}^*\cQ_{d-1}^R) \Bigr)\cong (\delta\times \id\times \id)_*\reg_{\Gamma_{d-1}}[-1]\cong \reg_B[-1]\,.\qedhere 
\]
\end{proof}

\begin{lemma}\label{eq:unionses}
 We have a short exact sequence
\begin{equation}\label{eq:Gammases}
  0\to \reg_A(-B)\to \reg_{\Gamma_d}\to \reg_T\to 0\,.
\end{equation}
\end{lemma}

\begin{proof}
We have $\Gamma_d=A\cup T$ and $B=A\cap T$. Furthermore, the projection $\pr_{12}\colon C\times C^d\times C\to C\times C^d$ maps $\Gamma_d$, $A$, and $T$ isomorphically to divisors in $C\times C^d$. Hence, 
\eqref{eq:Gammases} is the standard short exact sequence for unions of effective divisors (see also \cite[Eq.\ (7)]{Krug--stab} for this specific situation). 
\end{proof}

\begin{lemma}\label{lem:RFsect}
 We have $\Ho^0(\cS)\neq 0$.
\end{lemma}

\begin{proof}
 For $\pr_{13}\colon C\times C^d\times C\to C\times C$, we have $\pr_{13*}\cS\cong \wcQ_d\star \wcQ_d^R$, which is the Fourier--Mukai kernel of the composition of $\FM_{\wcQ_d}$ with its right-adjoint; see \autoref{lem:RF}.
 The unit of adjunction $\id_{\D(C)}\to (\FM_{\wcQ_d})^R\circ \FM_{\wcQ_d}$ is induced by a map $\reg_\Delta\to \wcQ_d\star \wcQ_d^R$ between the Fourier--Mukai kernels; see \cite[App.\ A]{CW}. Precomposing this map with the restriction map $\reg_{C\times C}\to \reg_\Delta$ gives a non-vanishing global section of $\wcQ_d\star \wcQ_d^R\cong \pr_{13*}\cS$, hence also of $\cS$. 
\end{proof}

\begin{proof}[Finishing the proof of \autoref{prop:FMpush}]
 In \autoref{lem:maincomp}, we computed all four corners of diagram \eqref{eq:convodiagbefore}, which gives
\begin{equation}\label{eq:convdiag}
\xymatrix{
\reg_B[-1] \ar[r] \ar[d]  &  [i_*\reg_p(1),\wcQ_d^R]  \ar[r] \ar[d]    &  \reg_A  \ar[d] \ar^{u}[r]  & \reg_B   \\
 [\wcQ_d,t^*\wcQ_{d-1}^R]  \ar[d]  \ar[r]  &  \cS \ar[d]  \ar[r]  &  [\wcQ_d,i_*\reg_p(-1)] \ar[d] &   \\
\reg_T   \ar^{\cong\qquad} [r] \ar^v[d]  & [t^*\wcQ_{d-1},\wcQ_{d}^R]   \ar[r]   &  0  &  \\
\reg_B &   &   &
 }
\end{equation}
For $d=1$ (necessary only if the reader did not already do this case themself as the base case), we have $B=T=\emptyset$. Hence $\reg_B\cong \reg_T\cong 0$, and \eqref{eq:convdiag} gives the assertion $\cS\cong \reg_A\cong \reg_{\Gamma_1}$.

So, from now on, let $d\ge 2$. 
We assume for a contradiction that $u\colon \reg_A \to \reg_B$ is the zero map, which would imply 
$[i_*\reg_p(1),\wcQ_d^R]\cong \reg_A\oplus \reg_B[-1]$. Taking the long exact cohomology sequence associated to the  
middle vertical exact triangle of \eqref{eq:convdiag} then gives
\begin{equation}\label{eq:convses}
 0\to \reg_A\to \cH^0(\cS)\to \reg_T\xrightarrow v \reg_B \to \cH^1(\cS)\to 0\,.
\end{equation}
As $B$ is projective and connected, the map $v\colon \reg_T\to \reg_B$ is either zero, or the restriction map (up to non-zero scalar multiplication). 
Let us first assume that $v=0$, which gives $\cH^1(\cS)\cong \reg_B$. By
\autoref{lem:convinva}, $\cS$ is a pull back of an object in $\D(C\times C^{(d)}\times C)$ along the $\sym_d$-quotient morphism
\[
\widehat \pi:=\id\times \pi\times \id\colon C\times C^d\times C\to C\times C^{(d)}\times C\,.
\]
As $\widehat \pi$ is flat (see \autoref{lem:piflat}), all  cohomology sheaves of $\cS$ must be a pull-back of some sheaf on $C\times C^{(d)}\times C$. In particular, \[\cH^1(\cS)\cong \reg_B\cong\widehat\pi^*_{\sym_d}\widehat\pi_*^{\sym_d}\reg_B\,;\]
see \autoref{cor:descent}. But this is not possible. Indeed, for $d>2$, the subvariety $B\subset C\times C^d\times C$ is not even $\sym_d$-invariant. For $d=2$, the reduced subscheme $B$ is $\sym_2$-invariant and connected. Hence, there are exactly two possibilities to equip $\reg_B$ with a $\sym_d$-linearisation; see \autoref{lem:twolin}. 
If $\reg_B$ is equipped with the canonical linearisation, then we have $\widehat \pi_*^{\sym_2}\reg_B\cong \reg_{\iota(C)}$ where
\[\iota\colon C\hookrightarrow C\times C^{(2)}\times C\quad,\quad \iota(x)=(x;2x;x)\] is the small diagonal. But since $\widehat \pi$ is flat of degree $2$, we find $\widehat \pi^* \reg_{\iota(C)}\cong \reg_{B'}$ for some \emph{non-reduced} subscheme $B'\subset C\times C^2\times C$ with $B'_{\mathsf{red}}=B$. In particular, $\widehat\pi^*_{\sym_d}\widehat\pi_*^{\sym_d}\reg_B\not \cong \reg_B$.
If $\reg_B$ is equipped with the anti-canonical linearisation instead, we have $\widehat\pi_*^{\sym_2}\reg_B\cong 0$ which also gives $\widehat\pi^*_{\sym_d}\widehat\pi_*^{\sym_d}\reg_B\not \cong \reg_B$.
In summary, we see that $v=0$ is impossible. 

Hence, (up to a non-zero scalar multiple) $v\colon \reg_T\twoheadrightarrow \reg_B$ is the restriction map. Then, $\cS$ is a sheaf concentrated in degree zero, and \eqref{eq:convses} becomes
\begin{equation}\label{eq:convses3}
 0\to \reg_A\to \cS\to \reg_T(-B)\to 0\,.
\end{equation}
As $T$ is connected, $\Ho^0(\reg_T(-B))=0$. Hence $\IC\cong \Ho^*(\reg_A)\cong \Ho^0(\cS)$. This means that every global section of $\cS$ is supported on $A$. But as $\cS$ carries a $\sym_d$-linearisation we also get an embedding $\tau^*\reg_A\cong \reg_{\tau(A)}\hookrightarrow \tau^*\cS\cong \cS$ where $\tau=(1\, 2)\in \sym_d$.
So, $\cS$ would also have a global section with support on the whole $\tau(A)$. This is a contradiction since $\tau(A)\not \subset A$.
This shows that our original assumption $u=0$ must have been wrong.

Hence, $u\colon \reg_A\twoheadrightarrow \reg_B$ is (up to non-zero scalar multiplication) the restriction map. So the the middle vertical triangle of \eqref{eq:convdiag} becomes the short exact sequence
\begin{equation}\label{eq:Tses}
 0\to \reg_A(-B)\to \cS\to \reg_T\to 0\,.
\end{equation}
Note that \eqref{eq:Tses} cannot split, as $\cS\cong \reg_A(-B)\oplus \reg_T$ would not be $\sym_d$-invariant (by essentially the same argument as the one above for the middle term of \eqref{eq:convses3} not to be invariant).
We have $\Ho^0(\reg_A(-B))=0$ and $\Ho^0(\reg_T)\cong \IC$. Hence, by \autoref{lem:RFsect}, $\Ho^0(\cS)\to \Ho^0(\reg_T)$ is an isomorphism. Let $s\in \Ho^0(\cS)$ correspond to $1\in \Ho^0(\reg_T)$ under this isomorphism. As $A\cup T=\Gamma_d$, the sequence \eqref{eq:Tses} shows that $s$ factors as
\[
 \reg_{C\times C^d\times C}\twoheadrightarrow \reg_{\Gamma_d} \xrightarrow{\bar s} \cS\,.
\]
Hence, we get a morphism of short exact sequences 
\[
\xymatrix{
 0\ar[r] & \reg_A(-B) \ar[r]\ar^{\alpha}[d]  & \reg_{\Gamma_d} \ar[r]\ar^{\bar s}[d]  &  \reg_T  \ar[r]\ar^1[d] & 0
 \\
0\ar[r] & \reg_A(-B) \ar[r]  & \cS \ar[r]  &  \reg_T  \ar[r] & 0
}
\]
where the lower sequence is \eqref{eq:Tses} and the upper sequence is the one from \autoref{eq:unionses}.
For $\alpha=0$, the lower sequence would split, what we ruled out above. 
Hence, $\alpha$ is an isomorphism, and the five lemma shows that $\bar s\colon \reg_{\Gamma_d}\to \cS$ is an isomorphism too.
\end{proof}

\begin{prop}\label{prop:FMSdC}For every $d\in \IN$, we have
\[
(\id\times \mu_d \times \id)_*\bigl(\pr_{12}^*\cQ\otimes \pr_{32}^*\cQ^R\bigr)\cong \reg_{\Gamma_{(d)}}\,. 
\] 
\end{prop}

\begin{proof}
For $d=1$, the map $\id \times\pi_1\times \id$ is the identity. So the assertion follows directly form \autoref{prop:FMpush} together with \autoref{lem:convinva}.  

For $d\ge 2$ there are points $z\in \Gamma_d$ for which some $g_z\in G_z$ has $\sgn(g_z)=-1$, for example $z=(x;x,\dots,x;x)$ for $x\in C$. Hence, by \autoref{lem:twolin}, the $\sym_d$-linearisation of  \[\alpha_{*}\bigl(\pr_{12}^*\wcQ_d\otimes \pr_{32}^*\wcQ_d^R\bigr)\cong \reg_{\Gamma_d}\] for which \eqref{eq:FMpull} and \eqref{eq:FMinva} in \autoref{lem:convinva} hold is the canonical one.  
Hence, the assertion follows by \eqref{eq:FMinva} and \eqref{eq:Gammainva}. 
\end{proof}

\begin{theorem}\label{thm:muHom} For every $d\in \IN$, and all $F,G\in \VB(C)$, we have
\[
 \mu_*\sHom(F^{\llb d\rrb}, G^{\llb d\rrb})\cong \sHom(F,G)^{[d]}\,.
\]
\end{theorem}

\begin{proof}
Considering the commutative diagram 
\[
 \xymatrix{
&\ar_{\pr_{13}}[dl] C\times \Quot_d\times C \ar^{\qquad\pr_2}[r]  \ar^{\id\times \mu\times \id}[d] & \Quot_d \ar_{\mu}[d]  \\
C\times C& \ar^{\pr_{13}\quad}[l] C\times C^{(d)}\times C \ar^{\quad\pr_2}[r] &  C^{(d)}\,,
  }
\]
we get 
\begin{align*}
 &\mu_*\sHom(F^{\llb d\rrb}, G^{\llb d\rrb})\\
 \cong & \mu_*\pr_{2*}\Bigl(\pr_1^*G\otimes \pr_{12}^*\cQ_d\otimes \pr_{32}^*\cQ_d^R\otimes \pr_3^* F^\vee  \Bigr)\tag{by \autoref{cor:HomtautFM}}\\  
\cong & \pr_{2*}(\id\times \mu \times \id)_*\Bigl(\pr_1^*G\otimes \pr_{12}^*\cQ_d\otimes \pr_{32}^*\cQ_d^R\otimes \pr_3^* F^\vee  \Bigr)\tag{commutativity of the square}\\
\cong & \pr_{2*}\Bigl(\pr_1^*G\otimes  (\id\times \mu \times \id)_*\bigl(\pr_{12}^*\cQ_d\otimes \pr_{32}^*\cQ_d^R\bigr)\otimes \pr_3^* F^\vee  \Bigr)\tag{commutativity of the triangle}\\ 
\cong & \pr_{2*}\Bigl(\pr_1^*G\otimes  \reg_{\Gamma_{(d)}}\otimes \pr_3^* F^\vee  \Bigr)\tag{by \autoref{prop:FMSdC}}\,.
\end{align*}
Now, recall from \autoref{subsect:Xi} that $\reg_{\Gamma_{(d)}}\cong \iota_{(d)*}\reg_{\Xi_d}$. Following the notation of \autoref{subsect:SdC}, we have 
\[
 (\pr_1\circ \iota_{(d)})_{\mid \Xi}=(\pr_3\circ \iota_{(d)})_{\mid \Xi}=a\colon \Xi \to C\quad,\quad (\pr_2\circ \iota_{(d)})_{\mid \Xi}=b\colon \Xi \to C^{(d)}   
\]
So, finally, the projection formula along $\iota_{(d)}$ gives
\[
 \pr_{2*}\Bigl(\pr_1^*G\otimes  \reg_{\Gamma_{(d)}}\otimes \pr_3^* F^\vee  \Bigr)\cong b_*a^*(G\otimes F^\vee)\cong (G\otimes F^\vee)^{[n]}\cong \sHom(F,G)^{[n]}\,.\qedhere
\]
\end{proof}

\begin{theorem}\label{thm:convolution} For every $d\in \IN$, we have in $\D(C\times C)$ an isomorphism
\[
 \cQ_d\star \cQ_d^R\cong \reg_{\Delta}\otimes_\IC S^{d-1}\Ho^*(\reg_C)\,.
\]
\end{theorem}

\begin{proof}
 Since $\pr_{13}\colon C\times \Quot_d\times C\to C\times C$ factors as 
 \[
C\times \Quot_d\times C\xrightarrow{\id\times \mu_d\times \id} C\times C^{(d)}\times C \xrightarrow{\pr_{13}} C\times C\,,  
 \]
\autoref{prop:FMSdC} gives
\begin{align*}
\cQ\star \cQ^R\cong \pr_{13*}\bigl(\pr_{12}^*\cQ\otimes \pr_{32}^*\cQ^R\bigr)  \cong   \pr_{13*}(\id\times \mu_d \times \id)_*\bigl(\pr_{12}^*\cQ\otimes \pr_{32}^*\cQ^R\bigr)\cong  \pr_{13*}\reg_{\Gamma_{(d)}}\,. 
\end{align*}
Now, consider the commutative diagram 
\[
\xymatrix{
C\times C^{(d-1)} \ar^j[r] \ar_{\pr_1}[d] &   C\times C^{(d)}\times C \ar^{\pr_{13}}[d] \\
C  \ar^{\delta}[r] &    C\times C
} 
\]
where $j(x;x_1+\dots+x_{d-1})=(x;x_1+\dots+x_{d-1}+x;x)$ is the closed embedding with image $\Gamma_{(d)}$. We get 
\begin{align*}
\pr_{13*}\reg_{\Gamma_{(d)}}\cong \pr_{13*}j_*\reg_{C\times C^{(d-1)}}\cong \delta_*\pr_{1*} \reg_{C\times C^{(d-1)}}&\cong \delta_*\Bigl( \reg_C\otimes_{\IC} \Ho^*(\reg_{C^{(d-1)}})\Bigr)\\&\cong \reg_{\Delta}\otimes_{\IC} S^{d-1}\Ho^*(\reg_C)\,.\qedhere 
\end{align*}
\end{proof}

By \autoref{lem:RF}(ii), \autoref{thm:convolution} implies \autoref{thm:mainfunctor}.

\section{Further remarks}\label{sect:furtherrem}

\subsection{Adding natural line bundles}
Given a line bundle $M\in\Pic(C)$, the associated line bundle  $M_{(d)}:=\pi_*^{\sym_d} M^{\boxtimes d}$ on $\Pic(C^{(d)})$ is the descent of $M^{\boxtimes d}$ along $\pi\colon C^d\to C^{(d)}$, which means that $\pi^*M_{(d)}\cong M^{\boxtimes d}$.
We have the following generalisation of \autoref{prop:cohtautSdC}.

\begin{lemma}\label{lem:tauttwistedcoh}
For $F\in \VB(C)$ and $M\in \Pic(C)$, we have for every $d\in \IN$:
\begin{align}
\Ho^*\Bigl(C^{(d)},  M_{(d)}\Bigr)&\cong S^d\Ho^*(M)\,, \label{eq:Mcoh}   \\
 \Ho^*\Bigl(C^{(d)}, F^{[d]}\otimes M_{(d)}\Bigr)&\cong \Ho^*(F\otimes M)\otimes S^{d-1}\Ho^*(M)\,. \label{eq:tautMcoh}
\end{align}
\end{lemma}

\begin{proof}
By the construction of $M_{(d)}$, its cohomology is the $\sym_d$-invariant part of the cohomology of $\Ho^*(C^d, M^{\boxtimes d})\cong \Ho^*(M)^{\otimes d}$. This gives \eqref{eq:Mcoh}.
For \eqref{eq:tautMcoh}, note that $\pi_d\colon C^d\to C^{(d)}$ factors as
\[
 C^d\xrightarrow{\id_C\times \pi_{d-1}}C\times C^{(d-1)}\cong \Xi_d \xrightarrow{b} C^{(d)}\,.
\]
Hence, under the isomorphisms $C\times C^{(d-1)}\cong \Xi_d$, we have $b^*M_{(d)}\cong M\boxtimes M_{(d-1)}$, by the uniqueness of the descent. Now, projection formula gives
\[
 F^{[n]}\otimes M_{(d)}\cong b_*(F\boxtimes \reg_{C^{(d-1)}})\otimes M_{(d)}\cong b_* \bigl((F\otimes M)\boxtimes M_{(d-1)}\bigr)\,.
\]
Hence, using the $d-1$ case of \eqref{eq:Mcoh}, we get
\[
\Ho^*\Bigl(C^{(d)}, F^{[d]}\otimes M_{(d)}\Bigr)\cong \Ho^*\Bigl(C\times C^{(d-1)}, (F\otimes M)\boxtimes M_{(d-1)}\Bigr)\cong \Ho^*(F\otimes M)\otimes S^{d-1}\Ho^*(M)\,.\qedhere
\]
\end{proof}
We define the \emph{natural} line bundle on $\Quot_d$ induced by $M\in \Pic(C)$ by
\[
M_{\llp d\rrp}:=\mu_d^* M_{(d)}\,.
\]
All our formulas for cohomologies and extension groups follow from formulas for the push forwards along $\mu_d\colon \Quot_d(E)\to C^{(d)}$. Hence, using the projection formula, we get slightly more general formulas  by adding natural line bundles in \eqref{eq:cohformula}, \eqref{eq:cohdualformula}, and \eqref{eq:Extformula}.

\begin{theorem}
Let $E$ be a vector bundle of rank at least 2 over a smooth projective curve $C$ and $d\in \IN$.
\begin{enumerate}
 \item For every $F\in \VB(C)$ and every $M\in \Pic(C)$, we have
\[
 \Ho^*\left(\Quot_d(E), F^{\llbracket d\rrbracket}\otimes M_{\llp d\rrp}\right)\cong \Ho^*(E\otimes F\otimes M)\otimes S^{d-1}\Ho^*(M)\,.
\]
\item Let $F_1,\dots, F_m$ be vector bundles on $C$, and let $1\le k_i\le d\rk F_i=\rk F_i^{\llb d \rrb}$ such that
$\sum_{i=1}^m\min\{k_i, \rk F_i\}<\rk E$. Then, for every $M\in \Pic(C)$, we have
\[
 \Ho^*\left(\Quot_d(E),M_{\llp d\rrp}\otimes \bigotimes_{i=1}^m\wedge^{k_i} F_i^{\llb d\rrb\vee}\right)\cong 0\,.
\]
\item For all $F,G\in \VB(C)$ and all $K,M\in \Pic(C)$, we have
\[
  \Ext^*\left(F^{\llb d\rrb}\otimes K_{\llp d\rrp},G^{\llb d\rrb}\otimes M_{\llp d\rrp}\right)\cong \Ext^*(F\otimes K,G\otimes M)\otimes S^{d-1}\Ext^*(K,M) \,.
\]
\end{enumerate}
\end{theorem}

\begin{proof}
 We only prove part (i), the other parts are similar. By the projection formula and \autoref{thm:maincoh}, we have
 \[
  R\mu_*\bigl(F^{\llbracket d\rrbracket}\otimes M_{\llp d\rrp}  \bigr)\cong \bigl( R\mu_*F^{\llbracket d\rrbracket}  \bigr) \otimes M_{(d)}\cong (E\otimes F)^{[d]}\otimes M_{(d)}\,.
 \]
Now, the assertion follows by \autoref{lem:tauttwistedcoh}.
\end{proof}

\subsection{Semi-orthogonal decomposition for $C=\IP^1$}

By \cite[Cor.\ 9.1]{BGS}, we have \[R\mu_{*} \reg_{\Quot_d}\cong \reg_{C^{(d)}}\] for every smooth projective curve $C$ and every $d\in \IN$. Hence, by projection formula, the functor $\mu^*\colon \D(C^{(d)})\to \D(\Quot_d)$ is fully faithful.  

For $C=\IP^1$ the projective line, we have $\Ho^*(\reg_{\IP^1})\cong \IC$. Thus, \autoref{thm:mainfunctor} gives in this special case $R_d\circ T_d\cong \id_{\D(C)}$ for all $d\in \IN$. Hence, the tautological functor \[T_d\colon \D(\IP^1)\to \D(\Quot_d)\] is fully faithful. By the $m=1$ case of \autoref{cor:dualvanish}(i), we have $R\mu_*\circ T_d\cong 0$. Hence, we have a semi-orthogonal decomposition
\begin{equation}\label{eq:sod}
\D(\Quot_d)=\langle \mu^*\D(\IP^{1(d)}), T_d \D(\IP^1), \,?\, \rangle\,. 
\end{equation}
In \cite{Toda--Quotsod}, for an arbitrary smooth projective curve $C$, Toda constructed a much finer semi-orthogonal decomposition than \eqref{eq:sod} by very different methods. 

As the variety $\Flag_d$ is an iterated projective bundle, it has semi-orthogonal decompositions with many components too. Hence, one may wonder whether the pull-back along $f\colon \Flag_d\to \Quot_d$ can somehow be used to give another proof of the semi-orthogonal decomposition of \emph{loc.\ cit}.,\ or maybe to construct another semi-orthogonal decomposition of $\Quot_d$.

\bibliographystyle{alpha}
\addcontentsline{toc}{chapter}{References}
\bibliography{references}

\end{document}